\documentclass[12pt]{amsart} 
\usepackage{amscd}
\usepackage{amssymb}
\usepackage{a4wide}
\usepackage{amstext}
\usepackage{amsthm}
\usepackage{xcolor}
\numberwithin{equation}{section}

\newcommand{\CC}{\mathbb {C}}

\newcommand{\RR}{\mathbb{R}}

 \DeclareMathOperator{\dist}{dist}

\DeclareMathOperator{\supp}{supp} 

\renewcommand{\phi}{\varphi}

\newcommand{\co}{\mathbb{C}}

\newcommand{\MM}{\mathcal{M}}

\newcommand{\rl}{\mathbb{R}}

\newcommand{\he}{\mathcal{H}(E)}
\newcommand{\ho}{\mathcal{H}}
\newcommand{\heo}{\mathcal{H}(E_1)}

\newcommand{\closspan}{\overline{\rm Span}\,}

\newtheorem{Thm}{Theorem}[section]
\newtheorem{theorem}[Thm]{Theorem}
\newtheorem{example}{Example}
\newtheorem{lemma}[Thm]{Lemma}
\newtheorem{proposition}[Thm]{Proposition}

\newtheorem{remark}[Thm]{Remark}
\newtheorem{definition}{Definition}

\begin{document}
\sloppy
\title[Localization of zeros for Cauchy transforms]{Localization of zeros for Cauchy transforms}
\author{Evgeny Abakumov, Anton Baranov, Yurii Belov}
\address{Evgeny Abakumov,
\newline  University Paris-Est,  LAMA (UMR 8050), UPEM, UPEC, CNRS, F-77454, Marne-la-Vall\'ee, France
\newline {\tt evgueni.abakoumov@u-pem.fr}
\smallskip
\newline \phantom{x}\,\, Anton Baranov,
\newline Department of Mathematics and Mechanics, St.~Petersburg State University, St.~Petersburg, Russia,
\newline National Research University Higher School of Economics, St.~Petersburg, Russia,
\newline {\tt anton.d.baranov@gmail.com}
\smallskip
\newline \phantom{x}\,\, Yurii Belov,
\newline Chebyshev Laboratory, St.~Petersburg State University, St. Petersburg, Russia,
\newline {\tt j\_b\_juri\_belov@mail.ru}
}
\thanks{A.~Baranov and Yu.~Belov were supported by the Chebyshev Laboratory
(St.\,Petersburg State University) under RF Government grant
11.G34.31.0026, by JSC "Gazprom Neft", and by RFBR grant 13-01-91002-ANF-a. 
A. Baranov and Yu.~Belov were  supported by  Dmitry Zimin's Dynasty Foundation.}

\begin{abstract} We study the localization  of zeros for Cauchy transforms of discrete measures  on the real line.
This question is
motivated by the theory of canonical systems of differential
equations. In particular, we prove that the spaces of Cauchy transforms having the
localization property are in one-to-one correspondence with the
canonical systems of special type, namely, those whose  Hamiltonians
consist only of indivisible intervals accumulating on the left.
Various aspects of the localization phenomena are studied in details. Connections with the density of polynomials and other topics in analysis are discussed.
%The case when we have only one accumulation point corresponds to
%the density of polynomial in the space $L^2(\mu)$,
%$\mu=\sum_n\mu_n\delta_{t_n}$. 
%In addition we describe measures
%$\mu$ which corresponds to the case of two accumulation points.
\end{abstract}

\keywords{Cauchy transforms, de Branges spaces, distribution of zeros of entire functions, canonical systems, polynomial approximation}
\subjclass{30D10, 30D15, 46E22, 41A30, 34B20}

\maketitle

\section{Introduction and main results}

The distribution of zeros of analytic and meromorphic functions is a classical theme in function theory.
Given a space of analytic functions, 
the information about the distribution of zeros of its elements is crucial for understanding the structural properties of the space.

Let $\mathcal{H}$ be an infinite-dimensional Hilbert space of
analytic functions in some domain $\Omega\subset \mathbb{C}$, with the division
property: $\dfrac{f(z)}{z-w}\in\mathcal{H}$ whenever
$f\in\mathcal{H}$ and $f(w)=0$. Thus, any finite subset $\Lambda$
of $\Omega $ is a non-uniqueness set for $\mathcal{H}$ (i.e., there
exists a nonzero function in $\mathcal{H}$ which vanishes on
$\Lambda$). Moreover, in many classical spaces (Hardy, Bergman,
Paley--Wiener, Bargmann--Fock, etc.) the zeros are even
less rigid in the sense that any sufficiently "sparse" countable
subset of $\Omega$ is also a non-uniqueness set for $\mathcal{H}$.

However,   we will show in the present paper that there exists a
natural class of spaces where the zeros of functions have  certain rigidity:
they are localized near some fixed discrete set. 
These are the spaces of  Cauchy transforms
of fast decaying measures.
% and associated de Branges spaces of entire functions. 
The goal of this paper is to study
various aspects of this localization phenomenon and to explore its
connections to several topics in analysis. These include the weighted
polynomial approximation, de Branges spaces of entire functions,
canonical systems of differential equations.
%and spectral theory of rank one perturbations of self-adjoint operators.
In particular, Theorem \ref{chain1} shows that the spaces of Cauchy transforms with the 
localization property are in one-to-one correspondence with the 
canonical systems having Hamiltonian of special form, namely,
those which consist of indivisible intervals accumulating only on
the left.

%**********************************************************

\subsection{The spaces of Cauchy transforms} Let
$\mu:=\sum_n\mu_n\delta_{t_n}$ be a positive measure on
$\mathbb{R}$ such that
$\int_{\mathbb{R}}\frac{d\mu(t)}{1+t^2}<\infty$. Here $T=\{t_n\}$
is an infinite increasing sequence (one or two-sided)  such that
$|t_n|\rightarrow\infty$,  $|n|\rightarrow\infty$. 
To simplify some formulas we assume that  $0\notin T$. With
any such $\mu$ we associate the space $\mathcal{H}(T,\mu)$ of the
Cauchy transforms
$$
\mathcal{H}(T,\mu):=\biggl{\{}f:f(z)=\sum_n\frac{a_n\mu^{1/2}_n}{z-t_n},
\quad a = \{a_n\}\in\ell^2\biggr{\}}
$$
equipped with the norm
$\|f\|_{\mathcal{H}(T,\mu)}:=\|a\|_{\ell^2}$.

The Cauchy transform is a classical object in complex analysis and
have numerous applications in various fields in mathematics (see
\cite{CR} and reference therein).

The spaces $\mathcal{H}(T,\mu)$ consist of analytic in
$\mathbb{C}\setminus T$ functions and satisfy the division property.
They were studied in \cite{BMS}. These spaces are {\it
isometrically isomorphic} in a canonical way to the Hilbert spaces
of entire functions introduced by de Branges. We will discuss the
details of this isomorphism below in Subsection \ref{dbr}.

%**********************************************************

\subsection{The localization and the strong localization} In what follows we will always assume that $T$ is 
{\it a power separated sequence}: there exist numbers $C>0$ and
$N\in\mathbb{N}$ such that
\begin{equation}
\label{powsep}
|t_{n+1}-t_n|\geq C|t_n|^{-N}.
\end{equation}
Note that condition \eqref{powsep} implies that for some $c,\rho>0$
and for sufficiently large $|n|$ we have $|t_n| \ge c|n|^\rho$.

Let us introduce some notation. For an entire function $f$ we
denote by $\mathcal{Z}_f$ the set of all zeros of $f$. For the
case
$f=\sum_n\frac{a_n\mu^{1/2}_n}{z-t_n}\in\mathcal{H}(T,\mu)$
it will be convenient to modify slightly the definition of the
zero set $\mathcal{Z}_f$. Namely, put
$\mathcal{Z}_f=\{w\in\mathbb{C}\setminus T: f(w)=0\}\cup\{t_n\in
T: a_n=0\}$. Finally let $D(z,r)$ stand for the open disc
centered at $z$ of radius $r$.

Our first result shows that several natural forms of localization are equivalent for the spaces  $\mathcal{H}(T,\mu)$.

\begin{theorem}
Let $\mathcal{H}(T,\mu)$ be a space of Cauchy transforms  with a 
power separated $T$. The following statements are equivalent:
\begin{enumerate}
\begin{item}
There exists an unbounded set $S\subset \mathbb C$ such that the set
$\mathcal{Z}_f\cap S$ is finite for any nonzero
$f\in\mathcal{H}(T,\mu)$;
\end{item}
\begin{item}
The set $\mathcal{Z}_f\setminus\cup_n D(t_n,1)$ is finite for any
nonzero $f\in \mathcal{H}(T,\mu)$;
\end{item}
\begin{item}
There exists a sequence of disjoint disks $\{D(t_n,r_n)\}$ such
that for any nonzero $f\in \mathcal{H}(T,\mu)$ the set
$\mathcal{Z}_f\setminus\cup_n D(t_n,r_n)$ is finite and each disk
$D(t_n,r_n)$ contains at most one point of $\mathcal{Z}_f$ for any
$n$ except, possibly, a finite number;
\end{item}
\begin{item}
There is no nonzero $f\in\mathcal{H}(T,\mu)$ with infinite number
of multiple zeros.
\end{item}
\end{enumerate}
\label{local}
\end{theorem}
This theorem leads us to the following definition.

\begin{definition} We say that the space $\mathcal{H}(T,\mu)$
with a power separated sequence $T$ has the localization property if
one of the equivalent conditions of Theorem \ref{local} holds.
\end{definition}

Theorem \ref{local} shows that if the zeros of functions from
$\mathcal{H}(T,\mu)$ are localized near some nontrivial subset of
$\mathbb{C}$ than they are also localized near some subset of $T$.
In this paper we study the structure of such subsets of $T$ which will be called {\it attraction sets}, see Subsection \ref{AS}.

For some spaces the zeros are localized only near the whole set $T$.

\begin{definition}
We say that the space $\mathcal{H}(T,\mu)$ with a power separated
sequence $T$ has the strong localization property if there exists
a sequence of disjoint disks $\{D(t_n,r_n)\}_{t_n\in T}$ such that for any
nonzero $f\in \mathcal{H}(T,\mu)$ each disk $D(t_n,r_n)$ contains
exactly one point of $\mathcal{Z}_f$ for any $n$ except, possibly,
a finite number.
\end{definition}

%It is easy to check (see, ...) that in this case any nonzero
%function $f\in\mathcal{H}(T,\mu)$ has only finite number of zeros
%outside the disks $D(t_n,r_n)$. 

It turns out that the strong
localization property is closely related to the approximation by
polynomials on $\mathbb{R}$.

\begin{theorem}
\label{strlocal}
The space $\mathcal{H}(T,\mu)$ has the strong localization
property if and only if the polynomials belong to
$L^2(\mu)$ and are dense there.
\end{theorem}

Note that polynomials belong to $L^2(\mu)$ whenever $\mathcal{H}(T,\mu)$  has the localization property (see Proposition \ref{mu_n}). Further we will identify the space $L^2(T, \mu)$ with the
weighted sequence space $\ell^2(\mu)$:
$$
\ell^2(\mu):=\bigl{\{}\{a_n\}:\sum_n|a_n|^2\mu_n<\infty\bigr{\}}.
$$

The question of density of polynomials in weighted $L^p$ spaces
is a famous longstanding problem in analysis. In the general
setting this problem was studied by M.~Riesz, S.~Bernstein, N.~Akhiezer, 
S.~Mergelyan, L.~de~Branges, and many
others. For an extensive discussion see, e.g., survey papers
\cite{Akh,Merg} and the monograph \cite{Koo1}. For the case of
measures supported by discrete subsets of $\mathbb{R}$ this
problem was recently investigated by 
A.~Borichev and M.~Sodin in \cite{BS0, BS1}. Further results 
about completeness of polynomials (as well as functions of exponential 
type and other classes of functions) were obtained by A.~Bakan \cite{bak},
Borichev and Sodin \cite{BS2}, A.~Poltoratski \cite{polt, polt1}, 
A.~Baranov and H.~Woracek \cite{bw}.

This gives us numerous examples of spaces with or without the
strong localization property.

\begin{example}
Let $T=\mathbb{Z}$ and let $\mu_n$ be an even sequence decreasing
on $\mathbb{N}$. Then the space $\mathcal{H}(T,\mu)$ has the
strong localization property if and only if
$$\sum_n\frac{\log\mu_n}{1+n^2}>-\infty.$$
\end{example}

As we will see, there exists a wide class of spaces with the
localization property for which the strong localization fails.
Such spaces   appear naturally  in the context of spectral theory of
canonical systems. This will be discussed in Subsection
\ref{dbr}. One of the simplest examples of such spaces is the
following:
\begin{example}
Put $\mu=\sum_{n\in\mathbb{N}}2^{-n(n-1)/2}n^2\delta_{2^n}$.
Then the space $\mathcal{H}(T,\mu)$ has the localization property
but not the strong one. \label{simex}
\end{example}

However under some smoothness condition on $\mu_n$ the
localization property implies the strong localization property,
see Theorem \ref{converse} in Subsection \ref{force}.

%**********************************************************

\subsection{Attraction sets}
\label{AS}
Let $\mathcal{H}(T,\mu)$ have the localization property. By the
property (iii) from Theorem \ref{local}, with any nonzero
function $f\in\mathcal{H}(T,\mu)$ we may associate a set
$T_f\subset T$ such that for some disjoint disks $D(t_n,r_n)$
all zeros of $f$ except, may be, a finite number
are contained in $\cup_{t_n\in T} D(t_n,r_n)$ and 
there exists exactly one point of $\mathcal{Z}_f$ in each disk
$D(t_n,r_n)$, $t_n\in T_f$, except, may be, a finite number of indices $n$.
Thus, the set $T_f$ is uniquely defined by $f$ up to finite sets.
Let us also note that we can always take $r_n=|t_n|^{-M}$ for
any $M>0$ (see condition (ii') in the beginning of Section \ref{locsec}).

\begin{definition}
Let $\mathcal{H}(T,\mu)$ have  the localization property. We will
say that $S\subset T$ is an attraction set if there exists
$f\in\mathcal{H}(T,\mu)$ such that $T_f=S$ up to a finite set.
\end{definition}

{Note that, in view of our definition of $\mathcal{Z}_f$, $f\in\mathcal{H}(T,\mu)$, for $f(z):=\frac{1}{z-t_0}$, $t_0\in T$, we have $\mathcal{Z}_f=T\setminus\{t_0\}$. So, $T$ is always an attraction set.}

It turns out that the localization property implies the following
ordering theorem for the attraction sets of $\mathcal{H}(T,\mu)$.

\begin{theorem}
Let $\mathcal{H}(T,\mu)$ be a space of  Cauchy transforms with
the localization property. Then for any two attraction sets $S_1$,
$S_2$ either $S_1\subset S_2$ or $S_2\subset S_1$ up to finite
sets. \label{ordering}
\end{theorem}

This ordering rule has some analogy with the de Branges Ordering
Theorem for the chains of de Branges subspaces. It is natural to
classify the spaces $\mathcal{H}(T,\mu)$ with respect to the
number of attraction sets. By $\#E$ we will denote the number of elements
in the set $E$.

\begin{definition} We say that the space $\mathcal{H}(T,\mu)$ has
the localization property of type $N$ if there exist  $N$
subsets $W_1$, $W_2$,...,$W_N$ of $T$ such that $W_j\subset W_{j+1}$, $1\leq j\leq N-1$,
$\#(W_{j+1}\setminus W_j)=\infty$ and for any nonzero
$f\in\mathcal{H}(T,\mu)$ we have $T_f=W_j$ for some $j$, $1\le j\le N$, 
up to  finite sets, moreover, $N$ is the smallest integer with this property.
\end{definition}
Clearly, $W_N=T$ up to 
a finite set. The strong localization is the localization of type $1$.

{
The notion of a Hamburger class function will be of  importance in what follows.
Following \cite{BS1} we say that an entire function $B$ of zero exponential type
(which is not a polynomial) 
belongs to the {\it Hamburger class} 
if it is real on $\RR$, has only real and simple zeros 
$\{s_k\}$, and for any $M>0$, $|s_k|^M = o(|B'(s_k)|)$, $s_k \to \infty$. }

As we will see (see Lemma \ref{A2M}), in the case of the localization of type $N$
the set $T\setminus W_j$ is   small in a sense for each $j$;
namely, it is the zero set of a Hamburger class function.
The next theorem provides a complete description of the spaces
with the  localization property of type $2$.

\begin{theorem}
 \label{type2}
The space $\mathcal{H}(T,\mu)$ has the localization property of
type $2$ if and only if there exists a partition $T=T_1\cup T_2$,
$T_1\cap T_2=\emptyset$, such that the following three conditions
hold:
\begin{enumerate}
\begin{item}
There exists a Hamburger class function $A_2$ such that
$\mathcal{Z}_{A_2}=T_2$;
\end{item}
\begin{item}
The polynomials belong to the space $L^2(T_2,{\mu }|_{T_2})$, 
are not dense there, but their 
closure is of finite codimension in $L^2(T_2,{\mu }|_{T_2})$.
\end{item}
\begin{item}
The polynomials belong to the space $L^2(T_1,\tilde{\mu})$ and are dense there,
where $\tilde{\mu}=\sum_{t_n\in T_1}\mu_n|A_2'(t_n)|^2\delta_{t_n}$.
\end{item}
\end{enumerate}
{Moreover, $T_1$ and $T$ are the attraction sets for $\mathcal{H}(T,\mu)$.}
\end{theorem}

These conditions mean that the measure $\mu$ consists of two 
essentially different parts $\mu = \mu|_{T_1}+ \mu |_{T_2}$.
The set $T_2$ is always rather sparse (as the  zero set of a Hamburger class function)
and the measure $\mu |_{T_2}$ is near the border of the polynomial density.
One may often (though not always) take $\mu_n = |t_n|^K |A_2'(t_n)|^{-2}$ 
with some $K>0$. On the other hand, the measure $\mu |_{T_1}$ is 
much smaller since the polynomials are dense 
not only in $L^2(T_1, \mu|_{T_1})$, but also in the space
$L^2(T_1,|A_2'|^2\mu |_{T_1})$, where $|A_2'(t_n)|$ tends to infinity 
faster then any polynomial, $|t_n|\to\infty$. The space from
Example \ref{simex} has the localization property of 
type $2$ and corresponds to the trivial partition $T=T_1 \cup T_2$,  
$T_1=\emptyset$, $T_2= T$.

We are able to give an analogous description for the spaces $\mathcal{H}(T,\mu)$ having the localization property
of type $N$ for any $N>2$ (see Theorem \ref{Nloc}).

A criterion of the polynomial density for the discrete measures 
supported by the zero set of a Hamburger class function was found by A. Borichev and M. Sodin in
\cite{BS1}. It is related to the description of the canonical 
measures for an indeterminate moment problem and, in particular,
with a curious mistake of Hamburger which remained unnoticed for 
about fifty years. We discuss these subjects in Subsection \ref{conc}. Note that
using the Borichev--Sodin criterion one can give a certain description
of the measures satisfying (ii) of Theorem \ref{type2}.

%**********************************************************
 
\subsection{De Branges spaces\label{dbr}}
An entire function $E$ is said to be in the Hermite--Biehler class
if $E$ has no real zeros and $|E(z)| >|E^*(z)|$,  
$z\in {\mathbb{C}_+}$, where $E^* (z) = \overline {E
(\overline z)}$. With any such function we associate the {\it de
Branges space} $\mathcal{H} (E) $ which consists of all entire
functions $F$ such that $F/E$ and $F^*/E$ restricted to
$\mathbb{C_+}$ belong to the Hardy space $H^2=H^2(\mathbb{C_+})$.
The inner product in $\he$ is given by
$$
(F,G)_{\he} = \int_\rl \frac{F(t)\overline{G(t)}}{|E(t)|^2} \,dt.
$$
There exist equivalent definitions of de Branges spaces. One of 
them is axiomatic (see \cite[Theorem 23]{br}). 
One more definition is related to the spaces of Cauchy transforms.
Let
$\mu$ be a positive measure on $\mathbb{R}$ as in the definition of
$\mathcal{H}(T,\mu)$ (i.e., $\mu=\sum_n\mu_n\delta_{t_n}$,
$|t_n|\rightarrow\infty$,
$\int_\mathbb{R}\frac{d\mu(t)}{1+t^2}<\infty$) and let $A$ be a
Weierstrass canonical product with zero set $T$. Then the space
$A\mathcal{H}(T,\mu)$  with the norm inherited from
$\mathcal{H}(T,\mu)$ is a de Branges space and, vice versa, any de
Branges space $\he$ can be represented in this way. The measure
$\mu$ is called {\it the spectral measure} for 
$\he = A\mathcal{H}(T,\mu)$.

The de Branges spaces play an important role in both complex
analysis and mathematical physics. They are  the crucial tool
in de Branges' celebrated solution of the inverse spectral problem
for canonical systems of differential equations (in particular,
for Schr\"{o}dinger equations on an interval).

May be the most spectacular fact in the de Branges theory is the
ordering theorem \cite[Theorem 35]{br} for the de Branges subspaces  
(those subspaces of $\he$ which are themselves de Branges spaces
with respect to the inherited norm). This theorem, in particular,
states that if $\mathcal{H}_1$ and $\mathcal{H}_2$ are two de
Branges subspaces of a given space $\he$, then either
$\mathcal{H}_1\subset\mathcal{H}_2$ or
$\mathcal{H}_2\subset\mathcal{H}_1$. However, given a de Branges space, 
there is no explicit way to reconstruct its chain. The possibility 
of the reconstruction of the properties 
of the subspaces from the properties of the final space in the 
chain is one of the deepest questions in the de Branges theory. 

We will say that $\he$ has the localization property if it satisfies either 
condition (i) or condition (iv) of Theorem \ref {local}. 
This is equivalent to say that the corresponding 
space  $\mathcal{H}(T,\mu)$ has the localization property.

It may happen that $\he$ contains a de Branges subspace of
codimension $1$. This condition is natural from the point of view   of the operator
theory and can be reformulated in many different ways, e.g., it is
equivalent to the  finiteness of {\it spectral measure} $\mu$, that is,
$\mu(\mathbb{R})<\infty$. If $\he$  has the localization property, then $\he$
has the finite spectral measure. Moreover, we have

\begin{theorem}
The de Branges space $\he$ \textup(or, equivalently $\mathcal{H}(T,\mu)$\textup)
with power separated $T=\supp\mu$ has the localization property if
and only if any de Branges subspace $\mathcal{H}_1$ of $\he$ which is not one-dimensional has a
de Branges subspace of codimension $1$. \label{chain}
\end{theorem}

Let $\he$ have the localization property and let
$\{\mathcal{H}_j\}_{j=0}^\infty$ be a decreasing sequence of de Branges
subspaces starting with $\mathcal{H}_0:=\he$ such that $\dim
(\mathcal{H}_j\ominus\mathcal{H}_{j+1})=1$, $j\geq0$. The strong
localization property (i.e., the localization property of type $1$) corresponds to the situation when there is
no other de Branges subspaces, $\cap_{j\geq0}\mathcal{H}_j=\{0\}$.

\begin{theorem}
The de Branges space $\he$ \textup(or, equivalently $\mathcal{H}(T,\mu)$\textup)
with power separated $T$ has the strong localization property if
and only if any de Branges subspace of $\he$ is of finite codimension. \label{strongchain}
\end{theorem}

The localization property of type $2$ corresponds to the situation when the   subspace
$\mathcal G = \cap_{j\ge 0} \mathcal{H}_j$ is non-zero (hence, it is a de Branges subspace) and has the localization property of type $1$. 
Such spaces have two attraction sets: 
$T$ and $T_{\tilde{A}}$, where $\tilde{A}$ is an entire function which vanishes on the support of the spectral measure 
of the space $\mathcal G = \tilde A\mathcal{H}(\tilde T,\tilde\mu)$.

Further, the localization property of type $N$ corresponds to the situation when the   subspace
$\cap_{j\ge 0} \mathcal{H}_j$ is non-zero   and has the localization property of type $N-1$.

%Using this corollary it is not difficult to construct a wide class
%of spaces having the localization property where the strong localization
%fails.

%**********************************************************

\subsection{Canonical systems}
Important examples of de\,Branges spaces occur in the theory of canonical
(or Hamiltonian) systems of differential equations, see, e.g.,
\cite[Theorems 37,  38]{br}, \cite{gk}, \cite{hdw}. Let $H(t)$ be a
$2\times 2$-matrix valued function defined for $t\in[0,L]$, such
that $H(t)$ is real and nonnegative with $tr H\equiv1$, the
entries of $H(t)$ belong to  $L^1([0,L])$. We call an open interval
$I\subseteq[0,L]$ {\it $H$-indivisible}, if the restriction of $H$ on
$I$ is a constant degenerate matrix and this fails for any open
interval $J\supsetneq I$ (i.e., $H(t)$ is the projector on a fixed
vector $e$ for all $t\in I$).

{  With each {\it Hamiltonian} $H$ we associate the 
so-called {\it canonical system of differential equations}
\begin{equation}
\label{can}
        Y'(t)J=zY(t)H(t),\qquad
        t\in[0,L], \qquad J:=
        \begin{pmatrix}
            0 & -1\\
            1 & 0
        \end{pmatrix},
\end{equation}
where $Y(t)$ is a $2\times 2$-matrix valued function on 
$[0,L]$ and $z\in\mathbb{C}$ is the spectral parameter. 
Let $Y(t,z)$ denote the (unique) solution of the initial value problem
\eqref{can} with $Y(0,z)=Id$.}

A wide class of second order differential equations (e.g., 
Schr\"{o}dinger equation, Dirac system) can be realized as some
canonical system.

Put $(A_t(z),B_t(z)):=(1,0)Y(t,z)$, $t\in[0,L]$, and
$E_t(z):=A_t(z)-iB_t(z)$. Then the function $E_t$ is a Hermite--Biehler
function. Moreover, the chain of de Branges subspaces of the space
$\mathcal{H}(E_L)$ is given by $\mathcal{H}(E_t), \,t\in [0, L]$, {$t$ is not an inner point of $H$-indivisible interval.}

Let $L^2([0,L], H)$ be the space of 2-vectors equipped with the norm
$$
\|g\|^2_{L^2([0,L], H)}:=\int_0^L \langle H(t)g(t), {g(t)}\rangle dt,
$$                                
where $\langle\cdot, \cdot\rangle$ stands for the usual inner product in $\mathbb{C}^2$.
There exists the generalized Fourier transform $\mathfrak{F}$ which
is unitary from $L^2([0,L], H)$,  subject to some natural
factorization, to $\mathcal H (E_L)$.

%The Paley-Wiener spaces can be realized in this way. In fact, if
%$H(t)=Id$, $t\in[0,L]$, then $E_t(z)=e^{-itz}$.

The de Branges inverse spectral theorem states that 
{\it the mapping $H\mapsto E_L$ is the one-to-one
correspondence between the  canonical systems and the regular de
Branges spaces}.  A de Branges space $\he$ is called 
{\it regular} if it is closed under forming
difference quotients $F(z) \mapsto \frac{F(z)-F(w)}{z-w}, w\in \mathbb C$.

We will say that $H$ consists of indivisible intervals if the union
of $H$-indivisible intervals is of full measure on $[0,L]$. If,
moreover, any point $x\in(0,L]$  either belongs to some
indivisible interval or is the right end of some indivisible
interval, we will say that $H$-indivisible intervals {\it accumulate
only on the left}.

The following theorem describes 
the regular spaces with the (strong) localization property 
in terms of the corresponding Hamiltonians.
Since ${\rm dim}\, \mathcal{H}(E_{t''})/\mathcal{H}(E_{t'}) = 1$ 
if and only if $I=(t',t'')$ is an $H$-indivisible interval,
this result is a direct corollary of Theorems \ref{chain}
and \ref{strongchain}. 

\begin{theorem} 
\label{chain1}
Let $\he$ be a regular de Branges space such that support $T$ of its spectral measure $\mu$ is power separated. Let $H$ be the corresponding Hamiltonian.

\begin{enumerate}
\begin{item}
The space $\he$ has the localization property if and only if $H$
consists of indivisible intervals accumulating only on the left.
\end{item}

\begin{item}
The space $\he$ has the strong localization property if and
only if the $H$-indivisible intervals accumulate only to $0$ 
\textup(i.e., $[0,L]=\overline{\cup_{n=1}^\infty I_n}$, where indivisible intervals $I_{n}$ and
$I_{n+1}$ have a common endpoint\textup).
\end{item}
\end{enumerate}
\end{theorem} 

So, the localization property of type $1$ corresponds to the case when the only one accumulation point is $0$.
It is not difficult to prove that there are  exactly $N$ accumulation points 
of the indivisible intervals if and only if
the de Branges space $\mathcal H(E_L)$ has the localization property of type $N$, 
i.e., there exist exactly $N$ attraction sets.
In general situation the ordering structure of attraction sets is the same as the ordering structure of the accumulation points of the indivisible intervals.

\textbf{Notations.} We will denote by $\mathcal{P}$ the set of all polynomials.  Throughout this paper the notation $U(x)\lesssim V(x)$ (or, equivalently,
$V(x)\gtrsim U(x)$) means that there is a constant $C$ such that
$U(x)\leq CV(x)$ holds for all $x$ in the set in question, $U, V\geq 0$. We write $U(x)\asymp V(x)$ if both $U(x)\lesssim V(x)$ and
$V(x)\lesssim U(x)$.

%%%%%%%%%%%%%%%%%%%%%%%%%%%%%%%%%%%%%%%%%%%%%%%%%%%%%

\section{Equivalent forms of zeros' localization\label{locsec}}
{
In this section we will prove Theorem \ref{local}.
Note that the implication (iii) $\Longrightarrow$ (iv) is trivial
as well as the implication (ii) $\Longrightarrow$ (i) (take $S = \mathbb{C}
\setminus \cup_n D(t_n, 1)$). To show the equivalence of the four conditions
we will prove that (i) $\Longrightarrow$ (ii), (ii) $\Longleftrightarrow$ (iv)
and (ii)\& (iv) $\Longrightarrow$ (iii).

In the proof of these implications it will be convenient 
to work not in the space of the Cauchy transforms 
$\mathcal{H}(T,\mu)$, but in the associated de Branges space $\mathcal{H}
= A\mathcal{H}(T,\mu)$.

{ 
A sequence $\{z_k\}\subset \CC$ will be said to be {\it lacunary} 
if $\inf_k |z_{k+1}|/|z_k| >1$. A zero genus canonical product over a lacunary 
sequence will be said to be a {\it lacunary canonical product}.}

The following result will play an important role in what follows. 
We will often need to verify that a certain entire function belongs 
to the de Branges space $\mathcal{H} = A\mathcal{H}(T,\mu)$. 
The following criterion is a special case of \cite[Theorem 26]{br}:

\begin{theorem}[Theorem 26 from the de Branges' book] 
\label{t26}
Let $\mathcal{H} = A\mathcal{H}(T,\mu)$ be a de Branges space. 
An entire function $F$ belongs to $\mathcal{H}$ if and only 
if $F/A$ is a function of bounded type \textup(i.e., 
a ratio of two bounded analytic functions\textup)  both in $\CC^+$ and $\CC^-$, 
\begin{equation}
\liminf_{|y|\to \infty}\Big|\frac{F(iy)}{A(iy)}\Big| = 0,
\label{FiyA}
\end{equation}
and
\begin{equation}
\label{inc}
\sum_{n}\frac{|F(t_n)|^2}{|A'(t_n)|^2 \mu_n} <\infty.
\end{equation}
\end{theorem}   }

{ Let us briefly explain the idea of the proof of Theorem \ref{t26}.  Condition \eqref{inc} implies that we can write a Lagrange-type
interpolation series for $F/A$, that is, $\sum_{n}\frac{F(t_n)}{A'(t_n)(z-t_n)}$, 
while \eqref{FiyA} ensure that this series represents the function $F/A$ and no additional entire term is present.
Thus $F\in A\mathcal{H}(T,\mu)$. The necessity of conditions \eqref{FiyA} and \eqref{inc} follows from the definition of $\mathcal{H}(T,\mu)$.}
\medskip

We will show that in Theorem \ref{local} (ii) can be repalced by a stronger property (ii'):

\medskip

{(ii') \it for any $F\in \ho\setminus\{0\}$ and $M>0$ we have
$\#\{\mathcal{Z}_F \setminus \cup_n D(t_n, |t_n|^{-M})\} <\infty$.}

\medskip

{ 
{\bf (i) $\Longrightarrow$ (ii')}.
Assume that (ii') is not true. Then
for some $M>0$ there exists a nonzero function $F \in \mathcal{H}$
for which there exists an infinite number of zeros
$z\in \mathcal{Z}_F$ with $\dist(z, T)\ge |z|^{-M}$. 

Let $S$ be an unbounded set which satisfies (i).
Then we can choose two sequences $s_k\in S$
and $z_k \in \mathcal{Z}_F$ such that $2|z_k| \le |s_k| \le |z_{k+1}|/2$
and $\dist(z_k, T)\ge  |z_k|^{-M}$. Now put
$$
H(z) = F(z) \prod_k \frac{1-z/s_k}{1-z/z_k}.
$$
A simple estimate of the infinite products implies $|H(z)| 
\lesssim |z|^{M+1}|F(z)|$ for $|z|\ge 1$
and $\dist(z, \{z_k\})\ge |z_k|^{-M}/2$, in particular, for $z\in T$.
Now dividing $H$ by some polynomial $P$ of degree  $M+1$
with $\mathcal{Z}_F \subset \mathcal{Z}_F \setminus\{z_k\}$,
we conclude by Theorem \ref{t26}  that $\tilde H = H/P$ is in $\mathcal{H}$.
This contradicts (i) since $\mathcal{Z}_{\tilde H} \cap S$
is  an infinite set.   }

\bigskip

{\bf (ii)$\Longrightarrow$(iv)} 
Assume that (iv) is not true. Then there exist a nonzero 
function $F\in\ho$ and a sequence $\{z_n\}$, $z_n = x_n +iy_n$, 
of its multiple zeros such that
$\dist(\{z_n\}, T)\leq 1$, $\inf_n{|z_{n+1}|/|z_n| > 2}$.
Put
$$
h(z) = \biggl{[}\prod_n\frac{1-z/ x_n}{1-z/ z_n}\biggr{]}^2.
$$
It is not difficult to prove that the product converges, $\sup_{x\in
\rl} |h(x)|<\infty$ and $\sup_{y\in \rl} |h(iy)|<\infty$.
Indeed, since $|y_n|\le 1$, we have
$$
\biggl{|}\frac{1-iy/ x_n}{1-iy/ z_n}\biggr{|}^2=
\frac{x_n^2 +y^2}{x_n^2+(y-y_n)^2} \cdot \frac{x_n^2+y_n^2}{x_n^2}
\leq \biggl{(}1+\frac{4}{x_n}\biggr{)}
\biggl{(}1+\frac{y^2_n}{x^2_n}\biggr{)}.
$$
The series $\sum_n\frac{y_n^2}{x_n^2}$ and $\sum_n \frac{1}{x_n}$
converge, and so, $\sup_{y\in \mathbb{R}} |h(iy)|<\infty$. On the
other hand $\bigl{|}\frac{1-x/ x_n}{1-x/
z_n}\bigr{|}^2 \leq \frac{x_n^2+y_n^2}{x_n^2}$ for any
$x\in\mathbb{R}$.

It follows from Theorem \ref{t26} 
that the function $H =Fh$ is in $\ho$, and, clearly, $H$ 
has multiple real zeros. 
Thus, we can assume from the beginning that there exists
a sequence of {\it real multiple zeros} $z_n$ of $F$. If
there exist a big number $M\in \mathbb{N}$ and an infinite
subsequence $\{z'_n\}$ such that $\dist(z'_n, T) \geq C
|t_{k+1}-t_k|^{-M}$, $z'_n\in[t_k, t_{k+1}]$, then the function
$$
\dfrac{F(z)}{\prod_{k=1}^{M+2}(z-\lambda_k)}
\cdot\prod_n\dfrac{z-(z'_n+i)}{z-z'_n}
$$
(where $\lambda_k\notin\{z'_n\}$ are some zeros of $F$) is in $\ho$ and we get a contradiction with (ii'). If there is no
such $M$, then all multiple zeros are well-localized near
$\{t_n\}$, namely for each $k$ (but a finite number) there exists
a number $n_k$ such that $|z_k-t_{n_k}|\leq C_N |t_{n_k}|^{-N}$,
for any $N\in\mathbb{N}$. In this case, the function
$$
\tilde F(z) = \dfrac{f(z)}{\prod_{k=1}^{K+2}(z-\lambda_k)}
\cdot\prod_k\dfrac{(z-(z_k+i))(z-t_{n_k})}{(z-z_k)^2}
$$
is in $\ho$ for sufficiently big $K$ and we also get a
contradiction. To see that $\tilde F \in \ho$ note that $|\tilde
F(t_n)| \lesssim |F(t_n)|$ for sufficiently large $K$  
and also $|\tilde F(iy)| \lesssim |F(iy)|$, $|y|>1$, whence $\tilde F$ 
is in $\ho$ by Theorem \ref{t26}.
\bigskip

{\bf (iv) $\Longrightarrow$ (ii)}. 
Let $\{z_n\}$ be a sequence of zeros of
$F\in\ho\setminus\{0\}$ with the property $\dist(z_n, T)\geq 1$.
Without loss of generality we can assume that
$\inf_n{|z_{n+1}|/|z_n| > 2}$. Put
$$
h(z)=\prod_n\frac{1-z/ z_{2n+1}}{1-z/ z_{2n}}.
$$
From the standard estimates for infinite products we get
$|h(z)|\lesssim (1+|z|)$ when $\dist(z,\{z_n\})>1$. If $\lambda$
is a zero of function $f$, then the function
$\frac{F(z)}{z-\lambda}\cdot h(z)$ is in $\ho$ and has infinite
number of multiple zeros $\{z_{2n+1}\}$.
\bigskip

{\bf (ii')\&(iv) $\Longrightarrow$ (iii)}. 
Let us consider the disjoint
disks $D(t_n, c|t_n|^{-N})$. If there 
exists a nonzero function $F\in\ho$ with an infinite number
of zeros $\{z_n\}$ outside these disks, we
have a contradiction with (ii'). Assume now that all zeros of $F$, 
except a finite number, 
are localized in the disks $D(t_n, c|t_n|^{-M})$ for sufficiently
large $M$. If an infinite subsequence of disks $D(t_{n_k},
c|t_{n_k}|^{-M})$ contains two zeros $z_k$, $\tilde z_k$ of $F$,
then (passing to a sparser sequence if necessary) we can show that
the function
$$
\dfrac{F(z)}{\prod_{k=1}^{N+2}(z-\lambda_k)}
\cdot\prod_k\dfrac{(z-t_{n_k})^2}{(z-z_k)(z-\tilde z_k)}
$$
is in $\ho$ (again, apply Theorem \ref{t26}). \qed
\medskip

We finish this section by a simple remark.
\begin{remark}
If the space $\ho$ has no localization property, then there exist
a non-zero $F\in \ho$ and an entire function $U$ with lacunary zeros,
$$
U(z):=\prod_{k=1}^\infty\biggl{(}1-\frac{z}{u_k}\biggr{)},
\qquad |u_{k+1}|>10|u_k|,
$$
such that $F/V\in\ho$ for any divisor $V$ of $U$,
$V=\prod_{k\in \mathcal{N}}\bigl{(}1-\frac{z}{u_k}\bigr{)}$,
$\mathcal{N}\subset \mathbb{N}$. \label{Uremark}
\end{remark}

For the proof it is sufficient to take $F$ which does not satisfy (ii)
and construct $U$ as a zero genus product over a lacunary
sequence $u_k \in \mathcal{Z}_F$ with $\dist(u_k, T)\ge 1$.

%%%%%%%%%%%%%%%%%%%%%%%%%%%%%%%%%%%%%%%%%%%%%%%%%%%%%

\section{Localization and polynomial density}

This section is devoted to the proof of Theorem \ref{strlocal}.
In Subsection \ref{pd1} we show that 
the polynomial density implies the strong localization property. 
In Subsection \ref{pd2} we will prove the converse statement.

{
First of all we prove that  the localization property implies that $\mu_n$ decrease superpolynomially.
\begin{proposition} Let $\mathcal{H}(T,\mu)$ have the localization property. Then for any $M>0$
$$\mu_n\lesssim |t_n|^{-M}.$$
\label{mu_n}
\end{proposition}
\begin{proof}
Assume the converse. Then there exists $M>0$ and infinite subsequence $\{n_k\}$ such that $\mu_{n_k}\geq |t_{n_k}|^{-M}$.
Without loss of generality we can assume that $\{t_{n_k}\}$ is lacunary. Let $U$ be the lacunary product with zeros $t_{n_{10k}}$. From Theorem \ref{t26} we conclude that the function
$$A(z)U^3(z)\prod_{k}\biggl{(}1-\frac{z}{t_{n_k}}\biggr{)}^{-1}$$
belongs to the de Branges space $\mathcal{H}:=A\mathcal{H}(T,\mu)$. This contradicts to the property (iv) from Theorem \ref{local}.
\end{proof}
}
\subsection{Polynomial Density $\Longrightarrow$ Strong Localization Property\label{pd1}}
Let $f\in\mathcal{H}(T,\mu)\setminus\{0\}$. We claim that for any $M>0$ there exist $L>0$ and $R>0$ such that 
\begin{equation}
\inf\{|z|^L|f(z)|: \dist(z,T)\geq|z|^{-M}, |z|>R\}>0.
\label{LReq}
\end{equation} 
Assume the converse. Then there exist a function $f(z):=\sum_{n}\dfrac{d_n\mu_n}{z-t_n}$, $\{d_n\}\in\ell^2(\mu)$, and a sequence $z_j\rightarrow\infty$, $j\rightarrow\infty$, such that $\dist(z_j,T)\geq|z_j|^{-M}$ and $|f(z_j)|<|z_j|^{-j}$.
Since the polynomials are dense in $\ell^2(\mu)$ we can take $K$ to be the smallest nonnegative integer such that $\sum_nd_nt^K_n\mu_n\neq0$. Then we write
$$f(z_j)=\sum_n\frac{d_n\mu_n}{z_j-t_n}=\sum_{|t_n|<|z_j|/2}
\frac{d_n\mu_n}{z_j-t_n}+\sum_{|t_n|\geq|z_j|/2}
\frac{d_n\mu_n}{z_j-t_n}=\Sigma_1+\Sigma_2.$$
Let us estimate the sums $\Sigma_1$ and $\Sigma_2$ separately. Since $|z_j-t_n|\geq |z_j|^{-M}$ for any $n$, we have
$$|\Sigma_2|\leq |z_j|^M\sum_{|t_n|>|z_j|/2}|d_n|\mu_n=|z_j|^M\sum_{|t_n|>|z_j|/2}\frac{|d_n|}{|t_n|^{K+M+2}}|t_n|^{K+M+2}\mu_n$$
$$\leq\frac{2^{K+M+2}}{|z_j|^{K+2}}\|\{d_n\}\|_{\ell^2(\mu)}\cdot\|\{t_n^{K+M+2}\}\|_{\ell^2(\mu)}.$$
On the other hand,
$$
\Sigma_1
=\sum_{|t_n|<|z_j|/2}d_n\mu_n\biggl{(}\sum_{l=0}^{\infty}\frac{t^l_n}{z^{l+1}_j}\biggr{)}=
\sum_{|t_n|<|z_j|/2}d_n\mu_n\biggl{(}\sum_{l=0}^{K}\frac{t^l_n}{z^{l+1}_j}+r_{n,j}\biggr{)},$$
where $r_{n,j}\leq \dfrac{2|t_n|^{K+1}}{|z_j|^{K+2}}$. Hence,
$$
\begin{aligned}
 &
f(z_j)=\sum_n\frac{d_n\mu_n}{z_j-t_n}=\sum_{|t_n|<|z_j|/2}\sum_{l=0}^Kd_n\mu_n\frac{t^l_n}{z^{l+1}_j}+
O(|z_j|^{-K-2}) \\
& =
\sum_n\sum_{l=0}^Kd_n\mu_n\frac{t^l_n}{z^{l+1}_j}-\sum_{|t_n|\geq|z_j|/2}
\sum_{l=0}^Kd_n\mu_n\frac{t^l_n}{z^{l+1}_j}+O(|z_j|^{-K-2}) \\
& =\frac{\sum_nd_n\mu_nt^K_n}{|z_j|^{K+1}}+O(|z_j|^{-K-2}).
\end{aligned}
$$
We get a contradiction and so the claim is proved.

Thus, in particular, we have shown that for any $M>0$ all zeros 
of $f\in\mathcal{H}(T,\mu)\setminus\{0\}$ except, 
may be, a finite number, are in $\cup_n D(t_n,|t_n|^{-M})$. 
Therefore by Theorem \ref{local}, the space $\mathcal{H}(T,\mu)$ 
has the localization property, and so any disc $D(t_n,|t_n|^{-M})$ 
except a finite number contains at most one zero of $f$.

Now we show that the disk $D(t_k,|t_k|^{-M})$ contains 
exactly one point of $\mathcal{Z}_f$ if $|k|$ is sufficiently large. Put
$$g(z)=\sum_{n\neq k}\frac{d_n\mu_n}{z-t_n}.$$
Recall that $|f(z)|\geq c|z|^{-L}$ for $|z-t_k|=|t_k|^{-M}$ and sufficiently large $k$, where $L$ is the number from \eqref{LReq}.
Since $\mu_k=o(|t_k|^{-\tilde{L}})$, $|k|\rightarrow\infty$, for any $\tilde{L}>0$, we conclude that $|f(z)-g(z)|<c|z|^{-L}/2$
for $|z-t_k|=|t_k|^{-M}$, $|k|\geq k_0$.

Put $F=Af$, $G=Ag$. Then $F$, $G$ are entire and 
$|F-G|<|G|$ on $|z-t_k|=|t_k|^{-M}$, $|k|\geq k_0$. 
By the Rouch\'{e} theorem, $F$ and $G$ have the same 
number of zeros in $D(t_k,|t_k|^{-M})$, $|k|\geq k_0$. 
Since $G(t_k)=0$, we conclude that $F=Af$ has a zero in 
$D(t_k,|t_k|^{-M})$, $|k|\geq k_0$. The strong localization property is proved.

\subsection{Strong Localization $\Longrightarrow$ Polynomial
Density.\label{pd2}} This implication is almost trivial. Let $\{u_n\} \in
\ell^2$ be a nonzero sequence such that $\sum_{n}  u_n t_n^k
\mu_n^{1/2} = 0$ for any $k\in \mathbb{N}_0$. Consider the
function
$$
F(z) = A(z) \sum_n \frac{u_n \mu_n^{1/2} }{z-t_n}.
$$
Then $F$ belongs to the de Branges space $\ho
= \mathcal{H}(T,\mu)$ and since all the moments of $u_n$ are zero, it is
easy to see that $|F(iy)/A(iy)|=  o(|y|^{k})$ as $|y|\to\infty$ for any $k>0$.
On the other hand, since we have the strong localization property,
for any $M>0$ all but a finite number of zeros of $f$ lie in
$\cup_n D(t_n, r_n)$, where $r_n = |t_n|^{-M}$ and
$\#\bigl{(} \mathcal{Z}_f\cap D(t_n,r_n)\bigr{)} \leq 1$ for all indices $n$ except,
possibly, a finite number. 

Let $T_1$ be the set of those $t_n$ for which the corresponding 
disk $D(t_n,r_n)$ contains exactly one zero of $F$ (denoted by $z_n$
with the same index $n$)
and let $A=A_1A_2$ be the corresponding factorization of $A$,
where $A_1$ is some Hadamard product with zeros
in $T_1$ and $A_2$ is in this case just a polynomial. 
Put
$$
F_1(z) = A_1(z)\prod_{t_n \in T_1} \frac{z-z_n}{z- t_n}.
$$
We can choose $M$ to be so large that the above product converges, and, moreover, 
$|F_1(z)|\asymp |A_1(z)|$
when $\dist(z, T_1) \ge c|z|^{-N}/2$, 
$N$ being the constant 
from \eqref{powsep}. Then we can write $F = F_1F_2$, and it is easy to 
see that in this case $F_2$ is at most a polynomial. 
Thus, $|F(iy)|/|A(iy)| \gtrsim |y|^{-M}$, $y\to\infty$,
for some $M>0$, and we have got a contradiction.
\qed

\subsection{Forced strong localization for ``good" measures. \label{force}} 
The next theorem shows that under some regularity conditions on
$(T, \mu)$ even localization property (not the strong one!) implies
that the polynomials are dense in $L^2(\mu)$.

\begin{theorem}
\label{converse} 
Let $\tilde{T}=\{t_{n_k}\}$ be an infinite
subsequence of $T=\{t_n\}$ such that the polynomials are
belong to the space $L^2(T\setminus\tilde{T}, \mu|_{T\setminus\tilde{T}})$ and
are not dense there. Suppose that there exists a positive
function $\MM$ on $\mathbb{R}$ such that $\mathcal{M}(t_n)=\mu^{1/2}_n$ and
$\MM$ is a normal weight \textup(that is, $\log \MM(e^t)$ is a convex
function of $t$\textup). Then $\mathcal{H}(T,\mu)$ does not have the localization property.
\end{theorem}

As we will see in the proof of Theorem \ref{converse} there is a
general principle that the non-density of polynomials in $L^2(\mu)$
implies a certain majorization result in the corresponding de
Branges space. The next proposition shows how the majorization, in
its turn, implies the non-localization of zeros.

\begin{proposition}
Let $\he = A\mathcal{H}(T,\mu)$ be some de Branges spaces 
with the spectral measure $\mu$ such that $\mu_n = o(|t_n|^{-M})$ for any $M>0$. 
Assume that there exists a nonzero function $f\in\he$ such that, for some
infinite subsequence $n_k$ we have $f(t_{n_k}) = 0$ and
$$
|f(t)|\leq |E(t)|\cdot \MM(t),\qquad  t\in\mathbb{R},
$$
where $\MM\in L^\infty(\mathbb{R})$ and $\MM(t)\leq \mu^{{1/2}}_{n_k}$
for $|t-t_{n_k}|<|t_{n_k}|^{-N}$ for some $N>0$. Then $\he$ has no
localization property. \label{majorant}
\end{proposition}

\begin{proof}
Assume that the converse holds. Dividing by a polynomial, we may
assume without loss of generality that
$$
|f(t)|\leq \frac{|E(t)|\mu^{1/2}_{n_k}}{|t_{n_k}|^N},\qquad
|t-t_{n_k}|<\mu_{n_k},
$$
where $N$ is some large fixed number. Let us show that there
exists $c>0$ such that the function
$$
g(z)=f(z)\prod_k \dfrac{z-(t_{n_k}+ic\mu_{n_k})}{z-t_{n_k}}
$$
satisfies $|g(t_{n_k})|\lesssim
|E(t_{n_k})|\mu^{1/2}_{n_k}|t_{n_k}|^{-N}$. First, note that if we
apply the Poisson formula in the upper half-plane to the function
$f/E$ we obtain the standard estimates that
$|f(z)/E(z)| \leq C \mu^{1/2}_{n_k} |t_{n_k}|^{-N}$
when $|z-t_{n_k}|<\mu_{n_k}/2$ and $z\in\mathbb{C}^+$.

To estimate $f/E$ in the lower half-plane, we will need the
following simple lemma (whose prove we omit).

\begin{lemma}
\label{zerloc} If $\he$ has the localization property, then we 
have $\dist (t_n, \mathcal{Z}_E) \asymp \mu_n$.
\end{lemma}

It is not difficult to show that $B(z) = \dfrac{E^*(z)}{E(z)}$ is
an interpolating Blaschke product (up to a possible finite number
of multiple zeros), and using Lemma \ref{zerloc} we get $|B(z)|
\lesssim 1$, $|z-t_{n_k}|\leq c \mu_{n_k}$, $z\in \mathbb{C}^-$.
So, $|f(z)/E(z)| \lesssim \mu^{1/2}_{n_k}
|t_{n_k}|^{-N}$, when $|z-t_{n_k}|=\varepsilon\mu_{n_k}$ for some
sufficiently small $\varepsilon>0$. The same estimate remains true
for $g/ E$, since $|g(z)|\asymp |f(z)|$ on the circle
$|z-t_{n_k}|=\varepsilon\mu_{n_k}$. Hence, $|g(t_{n_k})|\lesssim
|E(t_{n_k})|\mu^{1/2}_{n_k}|t_{n_k}|^{-N}$. Put
$$
h(z)=g(z)\prod_k\dfrac{z-(t_{n_k}+i)}{z-(t_{n_k}+ic\mu_{n_k})}.
$$
Then
$$
\sum_k\biggl{|}\frac{h(t_{n_k})}{E(t_{n_k})}\biggr{|}^2\mu_n\asymp
\sum_k\biggl{|}\frac{g(t_{n_k})}{E(t_{n_k})}\biggr{|}^2\frac{1}{\mu_n}
\lesssim\sum_k\frac{1}{|t_{n_k}|^{2N}} <\infty
$$
for suficiently large $N$. We get the inclusion $h\in\he$, and $h$
does not satisfy the condition (ii).
\end{proof}

%**********************************************************

\subsection{Proof of Theorem \ref{converse}}
{\it Step 1.} Let $\{s_n\}$ be a power separated real sequence and
$\nu=\sum_n\nu_n\delta_{s_n}$ be a positive measure. Suppose that
$\mathcal{P}\subset L^2(\nu)$ and $\overline{\mathcal{P}}\neq
L^2(\nu)$. 
%Let $\Theta$ be an inner function with tha Clark
%measure $\sigma_{-1} = \nu$. 
By the non-density of polynomials, there
exists a nontrivial sequence $\{d_n\}\in \ell^2(\nu)$ such that
$$
\sum_n d_ns^k_n\nu_n=0,\qquad k\in\mathbb{N}_0.
$$
Put $f(z)=\frac{A(z)}{E(z)}\sum_n\frac{d_n\nu_n}{z-s_n}$. We have
$$
t^kf(t)-\frac{A(t)}{E(t)}\sum_n\frac{d_n\nu_n}{s_n-t}=
\frac{A(t)}{E(t)}\sum_nd_n\nu_n\frac{t^k-s^k_n}{t-
s_n}=0, \qquad k\in\mathbb{N}_0.
$$
So,
$$
|f(t)|\leq \inf_{k\in\mathbb{N}_0} \frac{1}{|t|^k}\cdot
\frac{|A(t)|}{|E(t)|} \cdot 
\biggl{|}\sum_n\frac{d_ns^k_n\nu_n}{s_n-t}\biggr{|},\qquad
t\in\mathbb{R}.
$$
If $\dist(t,\{s_n\})\geq  |t_n|^{-M}$ for some $M>0$, then
$$
|f(t)|\leq \inf_{k\in\mathbb{N}_0}\biggl{(}\frac{1}{|t|^k}
\sup_n|s_n|^k\nu^{1/2}_n\biggr{)}\cdot |t|^{\tilde{M}},
$$
for some $\tilde{M}$. 
%Dividing by polynomial we get that there
%exists nontrivial $f$ such that
%$$
%|f(t)|\leq
%\inf_{k\in\mathbb{N}_0}\biggl{(}\frac{1}{|t|^k}\sup_n|s_n|^k\nu^{1/2}_n\biggr{)},\qquad
%\dist(t,\{s_n\})\geq |t|^{-N}.
%$$

{\it Step 2.} Put $\{s_n\}=T\setminus\tilde{T}$,
$\nu=\mu\bigr{|}_{T\setminus\tilde{T}}$, and define $f$ as in Step 1. 
Then, clearly, $f$ vanishes on $\tilde T$, and, dividing $f$ by a polynomial
if necessary, we obtain that there exists a nontrivial function $f$ with                                                           
$$
\begin{aligned}
\log|f(t)| &\leq \log \inf_{k\in\mathbb{N}_0}\frac{1}{|t|^k}
\sup_{t_n\in T\setminus\tilde{T}}|t_n|^k\mu^{1/2}_n \\
& = \inf_{k\in\mathbb{N}_0}\biggl{[}\sup_{t_n\in
T\setminus\tilde{T}}\biggl{(}{1/2}\log\mu_n+
k\log|t_n|\biggr{)}-k\log|t|\biggr{]}.
\end{aligned}
$$
Denote by $G^\#$ the Legendre transform of $G$,
$G^\#(x)=\sup_{t\in\mathbb{R}}(xt-G(t))$. Put $G(s)=-{1/2}\log
\MM(e^s)$. Then
$$
\log|f(t)|\leq \inf_{k\in\mathbb{N}_0} \big(G^\#(k)-k\log|t|\big), \qquad
|t-t_{n_k}|<\frac{1}{100|t|^N}.
$$
If $|t|>1$, then
$$ 
\inf_{k\in\mathbb{N}_0}\big(G^\#(k)-k\log|t|\big)
\leq C+\log|t|+\inf_{s\in\mathbb{R}}(G^\#(s)-s\log|t|)=C\log|t|-(G^\#)^\#(\log|t|).$$
From convexity of $G$ we conclude that $(G^\#)^\#=G$. Hence,
$$
|f(t)|\leq C|t|(\MM(t))^{1/2}, \qquad
|t-t_{n_k}|<\frac{1}{100|t|^N},\qquad |t|>1.
$$
Using Proposition \ref{majorant} we get the result.
\qed

\subsection{Polynomial approximation on discrete subsets of $\RR$; Hamburger's mistake.\label{conc}}  
Even in the case when $T$ is a zero set of some function in a Hamburger class, 
the density of polynomials in $L^2(T, \mu)$, $\mu = \sum_n \mu_n \delta_{t_n}$, 
is a subtle property. Such measures appear naturally in the 
Nevanlinna parametrization of all solutions of an indeterminate Hamburger 
moment problem (see the discussion in \cite{BS0, BS1}).

Interestingly, the work of Hamburger on this topic contained a mistake which 
remained unnoticed for about  fifty years.                                   
In paper \cite{Ham} Hamburger claimed that the polynomials are dense 
in the space $L^2(T, \mu)$, 
where $T$ is the zero set of some function $A$ of Hamburger class, whenever
\begin{equation}
\label{serA}
\sum_{t_n\in T}\frac{1}{\mu_n|A'(t_n)|^2}=\infty, \qquad
\sum_{t_n\in T}\frac{1}{(1+t^2_n)\mu_n|A'(t_n)|^2}<\infty.
\end{equation}
This result was then applied to a description of canonical measures
in the Nevanlinna parametrization. 

If true, Hamburger's result would, in particular, imply that
for the measure $\mu = \sum_{t_n\in T} |A'(t_n)|^{-2}\delta_{t_n}$ 
the polynomials are dense in $L^2(T, \mu)$.
In 1989 a gap in the Hamburger argument was found by C. Berg and H.L. Pedersen, 
and soon P. Koosis \cite{Koo2} gave an example of a 
Hamburger class function $A$ such that 
the polynomials are not dense in the space
$L^2(T, \sum_{t_n \in T} |A'(t_n)|^{-2}\delta_{t_n})$. 
The set $T$ from his construction consists of pairs of points which 
are exponentially close to each other and, in particular, 
$T$ is not a power separated sequence. 

Let us present now a simple counterexample 
to Hamburger statement with power separated $T$. Let
$$
A(z):=\prod_{n=1}^\infty\biggl{(}1-\frac{z}{n^2}\biggr{)}
\cdot\prod_{n=1}^\infty\biggl{(}1-\frac{z}{n^3+1/2}\biggr{)}
$$
$T:=\mathcal{Z}_A$, and $\mu:=\sum_{t_n\in T} |A'(t_n)|^{-2} 
\delta_{t_n}$. Then 
$A$ is of Hamburger class and \eqref{serA} is satisfied,  
but it is easy to check that the function
$A_1(z):=\prod_{n=10}^\infty\bigl{(}1-\frac{z}{n^2}\bigr{)}$ 
is in $A\mathcal{H}(T,\mu)$.
To see this, apply as usual Theorem \ref{t26} 
and the fact that the product 
$\prod_{n=1}^\infty\bigl{(}1-\frac{z}{n^2}\bigr{)}$
is bounded on $[0, \infty)$. Hence, by Theorem \ref{strlocal}, 
the polynomials are not dense in $L^2(\mu)$ (of course, one can give 
a direct proof of this statement).

Borichev and Sodin gave a criterion for the density of 
polynomials in $L^2(T, \mu)$, where $T$ is the zero set of some Hamburger
class function $A$ (see \cite[Corollary 1.1]{BS1}): the polynomials are dense in $L^2(T, \mu)$
if and only if for any $\tilde A$ such that $\tilde A$ is of Hamburger class 
and $\mathcal{Z}_{\tilde A} \subset \mathcal{Z}_A$
one has
$$
\sum_{t_n \in \mathcal{Z}_{\tilde A}} \frac{1}{\mu_n |\tilde A'(t_n)|^2}<\infty.
$$

One can give a similar criterion for the property that
the polynomials are not dense in the space
$L^2(T, \mu)$ of the above form,  
but their closure is of finite codimension
(the condition which appears in Theorem \ref{type2}).

%**************************************************************
%**************************************************************

\section{Structure of de Branges spaces with the localization property}
This section contains the proof of Theorem \ref{chain}.

It is well known that a de Branges space 
$\ho$ contains a de Branges subspace of codimension 
$1$ if and only if its spectral measure $\mu$ is finite, 
$\mu(\mathbb{R})<\infty$ (see \cite{br}).
So, it is sufficient to prove that the localization property 
is equivalent to the finiteness of all spectral measures 
corresponding to all de Branges subspaces.

\subsection{Localization $\Longrightarrow$ Finiteness of the Spectral Measures} Let
$\nu=\sum_k\nu_k\delta_{s_k}$ be a spectral measure of some de Branges
subspace $\ho_1 = \ho(\tilde E)$ of $\ho=\he$. 
Then the support of $\nu$ is  power separated. 
This follows immediately  from the following estimate of 
the continuous branch of the argument of $\tilde E$ on
$\mathbb{R}$ (see, \cite[Problem 93]{br}):
$$
(\arg \tilde E )(b)-
(\arg \tilde E )(a) \leq \pi \text{ whenever } (\arg E)(b)-(\arg E)(a)=\pi
$$
and the fact that for two neighbor points $s_k$ and $s_{k+1}$
in the support of $\mu$ we have $(\arg \tilde E )(s_{k})-
(\arg \tilde E )(s_{k+1}) = \pi$ (note that $\arg \tilde E$ is decreasing on $\mathbb{R}$).

Assume the converse, that is, let 
$\nu(\mathbb{R})=\infty$. Since $\{s_k\}$ is a power separated sequence,
there exists a subsequence $\{n_k\}_{k=1}^\infty$ and $M>0$ such that
$\nu_{s_{n_k}}|s_{n_k}|^M \gtrsim 1$. We will assume that the positive
sequence $s_{n_k}$ is lacunary, $s_{n_{k+1}}>10s_{n_k}$. Put
$$
U(z)=\prod_{k=1}^\infty\biggl{(}1-
\frac{z}{s_{n_{10k}}+i}\biggr{)}, \qquad
V(z)=\prod_{k=1}^\infty\biggl{(}1- \frac{z}{s_{n_{k}}}\biggr{)}.
$$
It is not difficult to see (applying once again Theorem \ref{t26}) 
that the function 
$$
A_{E_1}(z)\cdot\frac{U^2(z)}{V(z)}
$$
belongs to $\he$. This contradicts the localization property
(iv) from Theorem \ref{local}.

%*************************************************************

\subsection{Finiteness of the Spectral Measures $\Longrightarrow$ Localization}  We
need the following two results.

\begin{proposition}
Let $\heo$ be a de Branges subspace of $\he$ and let $U$ be a lacunary
canonical product,
$$
U(z)=\prod_{k=1}^\infty\biggl{(}1-\frac{z}{u_k}\biggr{)},\qquad |u_{k+1}|>10|u_k|.
$$
If $F\in\heo$ and $F/U \in \he$, then $F/U \in\heo$.
\label{Udiv}
\end{proposition}

\begin{proposition}
Let $\he$ be a de Branges space such that its spectral measure is
finite. Then there exist a de Branges subspace $\heo$ of
codimension $1$. Moreover, if $F\in\he$ and $F(\lambda)=0$, then
$F(z)/(z-\lambda)\in\heo$. 
\label{fDiv}
\end{proposition}

Proposition \ref{Udiv} follows immediately from Theorem 26 in
\cite{br}. It easy to derive Proposition \ref{fDiv} from Theorem
29 in \cite{br}. Now we will use Remark \ref{Uremark}. We fix a
non-trivial function $F\in\he$ and a lacunary canonical product $U$
such that $F/V\in\he$ for any divisor $V$ of $U$ and $F\slash U$ has at least one zero $z_0$.

From the de Branges Ordering Theorem (see \cite[Theorems 35 and 40]{br})
we know that all subspaces are ordered by inclusion 
and, moreover, may be parametrized by a real parameter. We denote
this chain by $\mathcal{B}:=\{\mathcal{H}(E_x)\}_{x\in\mathcal{N}}$, $\mathcal{N}\subset(0,1]$,
$\heo=\he$. The subspace $\mathcal{H}(E_x)$ is isometrically
embedded in $\mathcal{H}(E_y)$ if and only if $x\leq y$, $x,y\in\mathcal{N}$.
% It may happen that subspaces $E_x$ coincide for different $x$. Moreover
%it always happens in our situation. 
From finiteness of
the spectral measure of $\mathcal{H}(E_x)$ we conclude that for any
$x\in(0,1]$ the space $\mathcal{H}(E_x)$ contains de Branges
subspace $\mathcal{H}(E_y)$, $y<x$, of codimension $1$. 
This means that $(y,x)\cap \mathcal{N}=\emptyset$.
%This means that for any $t\in[y,x]$ either $\mathcal{H}(E_t)=\mathcal{H}(E_y)$ 
%or $\mathcal{H}(E_t)=\mathcal{H}(E_x)$.
We conclude that $\mathcal{N}$ is at most countable, that is, there exist at most countable number of
different de Branges subspaces $\mathcal{H}(E_x)$.

Now, we represent $U$ as an infinite product of non-constant lacunary canonical products
$$
U(z)=\prod_{l\in \mathcal{N}}U_l(z).
$$
Put
$$
V^x(z)=\prod_{l \in \mathcal{N}, l>x}U_l(z).
$$ 
So, $V^0(z)\equiv U(z)$, $V^1(z)\equiv 1$.

\begin{lemma}
For any $x\in(0,1]$, $F/V^x\in\mathcal{H}(E_y)\text{ for any } y\geq x,\quad y\in\mathcal{N}$.
\end{lemma}

\begin{proof} Put $L=\{x:x \in(0,1], F/V^x(z)\in\mathcal{H}(E_y) \text{ for any } y\geq x,\quad y\in\mathcal{N}\}$ and $x= \inf L$.
Suppose $x>0$. If $x \in L$, then there exists $y<x$, such that
$[y,x] \subset L$. Indeed, we can consider a subspace
$\mathcal{H}(E_y)$ of codimension $1$ in $\mathcal{H}(E_x)$. From
Proposition \ref{Udiv} we conclude that $F/V^y=
F/(V^xU_x)\in \mathcal{H}(E_y)$ and, hence, $y\in L$
which is absurd.

On the other hand, if $x\not\in L$, then by the same arguments 
$f/V^x\in \mathcal{H}(E_t)$ for any $t>x$, $t\in\mathcal{N}$.
We know that there exists sequence $x_n\in L$ such that $x_n\rightarrow x$.
So, $\mathcal{H}(E_x)=\cap_{y>x, y\in\mathcal{N}}\mathcal{H}(E_y)$. We conclude
that $F/V^x \in \mathcal{H}(E_x)$ and, hence, $x\in L$.
\end{proof}

Now we consider the function $F/U$. {Using Proposition
\ref{Udiv} we get 
$$\frac{F}{U},\quad \frac{F}{(z-z_0)U} \in \bigcap_{x>0,x\in\mathcal{N}}\mathcal{H}(E_x).$$
On the other hand, $\dim  \cap_{x>0,x\in\mathcal{N}}\mathcal{H}(E_x)\leq 1$.}
This contradiction proves Theorem \ref{chain}.
\qed

%**************************************************************
%**************************************************************

\section{Ordering theorem for the zeros of Cauchy transforms}

We will give two different proofs of Theorem \ref{ordering}. 
In the first proof the de Branges Ordering Theorem
will play a crucial role. The second proof is more elementary.

\subsection{Preliminaries} 
\label{prel}
We will start with the following two lemmas 
which will be of importance for each of the proofs.

\begin{lemma}
\label{A2M}
Let the de Branges space $\he(=A\mathcal{H}(T,\mu))$  have the localization property and let $A=(E+E^*)/2$ be the function
with simple zeros exactly on $T$. If $A=A_1A_2$, where $A_1$ and
$A_2$ are entire and $A_1\in\he$, then $A_2$ is of zero
exponential type and
\begin{equation}
\sum_{t_n
\in\mathcal{Z}_{A_2}}\frac{1}{\mu_n|A_2'(t_n)|^2}<\infty.
\label{A2Meq}
\end{equation}
\end{lemma}

Note that we have $\mu_n=o(|t_n|^{-N})$,
whence $|t_n|^N=o(|A'_2(t_n)|)$ for any $N>0$. Thus, $A_2$ belongs
to the usual Hamburger class as defined in \cite{BS1}

Recall that a function $f$ is said to be {\it of bounded type} 
in the upper half-plane $\mathbb{C}^+$, if
$f =g/h$, where $g,h$ are bounded analytic functions in $\mathbb{C}^+$
(or functions in $H^p$). If, in addition, $h$ is outer, then $f$ is said to be in {\it the Smirnov class} $\mathcal{N}_+(\mathbb{C}^+)$.
\medskip
\\
{\bf Proof of Lemma \ref{A2M}.} 
Since $A_1\in\he$ and $\frac{A(z)}{z-t}\in\he$ for any $t\in T$,
the functions $A_1/ E$ and $A/ E$ belong to the Smirnov
class $\mathcal{N}_+(\mathbb{C}^+)$,
whence $A_2$ is of  bounded type in $\mathbb{C}^+$. Analogously (passing
to the conjugate functions) $A_2$  is of bounded type in $\mathbb{C}^-$
and so, by a theorem of M.G. Krein 
(see, e.g., \cite[Chapter I, Section 6]{hj}) $A_2$ is of finite
exponential type.
If the exponential type of $A_2$ is positive, then it follows from the Nevanlinna factorization of ${A_1}\slash{E}$ that for sufficiently small $\varepsilon>0$ either $A_1e^{i\varepsilon z}\in\he$ or $A_1e^{-i\varepsilon z}\in\he$. Assume  that $A_1e^{i\varepsilon z}\in\he$. Hence,  $A_1(e^{i\varepsilon z}-\alpha)\in\he$
for any $\alpha\in\mathbb{C}$. This contradicts the localization property of $\he$. So, $A_2$ is of zero exponential type.

Recall that if $F\in\he$ then
$$
\sum_{t_n\in T}\frac{|F(t_n)|^2}{\mu_n|A'(t_n)|^2}<\infty
$$
(see Theorem \ref{t26}). Applying this to $F=A_1$ we get \eqref{A2Meq}. 
\qed

Let us show that it suffices to prove the ordering theorem 
for functions with zeros in the set $T$ only.

\begin{lemma}
\label{TfLemma}
Let $f\in\he$, $f\neq0$, and let $T_f$ be defined as in Subsection
\ref{AS}. Then there exists a function $A_f\in\he$ which vanishes
exactly  on $T_f$ up to a finite set. 
\end{lemma}

\begin{proof}
Let $z_n$ be a zero of $f$ closest to the point $t_n\in T_f$.
Since $T_f$ is defined up to finite sets, we may 
assume without loss of generality  that this is a one-to-one
correspondence between $\mathcal{Z}_f$ and $T_f$. Put
$$
A_f(z)=f(z)\prod_{t_n\in T_f}\frac{z-t_n}{z-z_n}.
$$

Since we have $|z_n-t_n|\leq|t_n|^{-M}$ with $M$ much larger than
$N$ from the power separation condition \eqref{powsep}, 
it is easy to see that $|A_f(iy)|\asymp|f(iy)|$, $|y|\rightarrow\infty$, and
$$|A_f(t_n)|\asymp|f(t_n)|,\qquad t_n\in T\setminus T_f.$$ Hence,
$A_f\in\he$ by Theorem \ref{t26}.
\end{proof}

%***********************************************************

\subsection{First proof of Theorem \ref{ordering}}
{The following lemma will play a crucial role in the first
proof of Theorem \ref{ordering}.

\begin{lemma}
\label{subconstr}
Let $\he$ be a de Branges space which has the localization property.
If $f\in\he$ is real on the real line and has only real simple
zeros, then
\begin{equation}
\mathcal{F}:=\closspan_{\he}\biggl{\{}\frac{f(z)}
{z-\lambda}\biggr{\}}_{\lambda\in\mathcal{Z}_f}
\label{subsp}
\end{equation}
is a de Branges subspace of $\he$ \textup(that is,
$\mathcal{F}=\mathcal{H}(\tilde{E})$\textup) and $T_f=T_{\tilde{A}}$,
where $\tilde{A}$ is the corresponding $A$-function of the space
$\mathcal{H}(\tilde{E})$. 
\end{lemma}

\begin{proof}
It is easy to check that $\mathcal{F}$ satisfies all axioms of de
Branges spaces (see \cite[Theorem 23]{br}). So, $\mathcal{F}$ is a de Branges space
and has localization property. Hence, $T_f\subset T_{\tilde{A}}$.
It remains to show that $T_{\tilde{A}}\subset T_f$ (we remind that
$T_f$ and $T_{\tilde{A}}$ are defined up to finite sets). 

Assume the
converse. This means that there exists a factorization of
$\tilde{A}$, $\tilde{A}=\tilde{A}_1\tilde{A}_2$, such that
$\tilde{A}_1\in\mathcal{F}$, the function
$\tilde{A}_2$ has infinite number of zeros, 
and the zeros of $f$ are localized exactly near
$\mathcal{Z}_{\tilde{A}}$ (i.e., for any $M>0$,
we have $|z_n-\tilde t_n|<|\tilde t_n|^{-M}$ 
for sufficiently big $n$, where $z_n$
and $\tilde t_n$ are the zeros of $f$ and $\tilde{A}$ respectively). Without loss of
generality we can assume that $\mathcal{Z}_f\cup
\mathcal{Z}_{\tilde{A}}=\emptyset$. 

Now we want to construct
a nonzero function $h\in\mathcal{F}$ such that 
$h\perp \big\{ \frac{f(z)}{z-\lambda}\big\}_{\lambda\in\mathcal{Z}_f}$.
This will give us a contradiction.
We use an idea of the construction of such vector $h$ which 
goes back to \cite{bb}.

Let
$\tilde{\mu}:=\sum_{\tilde t_n \in\mathcal{Z}_{\tilde{A}}}
\tilde \mu_n\delta_{\tilde t_n}$
be the spectral measure of $\mathcal{F}$ and let $\tilde k_n$ be the 
reproducing kernel of the space $\mathcal{F}$ at the point $\tilde t_n$.
So,
\begin{equation}
\frac{\tilde k_n(z)}{\|\tilde k_n\|_{\mathcal{F}}}=
\frac{\|\tilde k_n\|_{\mathcal{F}}}{\tilde{A}'(\tilde t_n)}
\cdot\frac{\tilde{A}(z)}{z-\tilde t_n}.
\label{ort}
\end{equation}
It is well known that the system $\{\tilde k_n\}$ is an orthogonal system
in $\mathcal{F} = \mathcal{H}(\tilde{E})$ \cite[Theorem 22]{br}. Put
$$
h(z)=\sum_nh_n\frac{k_n(z)}{\|\tilde k_n\|_{\mathcal{F}}},
\qquad \{h_n\}\in\ell^2.
$$
Equation $(h, f(z)/(z-\lambda))_{\mathcal{F}}=0$,
$\lambda\in\mathcal{Z}_f$, 
is equivalent to
$$
\sum_n\frac{h_nf(t_n)}{\|\tilde k_n\|_{\mathcal{F}}(\lambda-t_n)}=0,\qquad \lambda\in\mathcal{Z}_f.
$$
This is equivalent to the interpolation formula
\begin{equation}
\sum_n\frac{h_nf(t_n)}{\|\tilde k_n\|_{\mathcal{F}} (z-t_n)}=
\frac{f(z)S(z)}{\tilde{A}(z)},
\label{finterp}
\end{equation}
where $S$ is some entire function. From equation \eqref{ort}
we conclude that
$\tilde \mu_n^{1/2} = \|\tilde k_n\|_{\mathcal{F}}/|\tilde{A}'(\tilde t_n)|$. 
So, the inclusion $\{h_n\} \in \ell^2$ is equivalent to
$\sum_n|S(\tilde t_n)|^2\mu_n<\infty$. From Theorem \ref{chain} we know
that the spectral measure $\tilde{\mu}$ is finite and, so, we can take
$S\equiv1$. It remains to show that the interpolation formula
\eqref{finterp} holds.

Clearly, the difference between the right-hand side and the left-hand side in 
\eqref{finterp} is an entire function of zero exponential type. 
It remains to show that this difference tends to 0 along the imaginary axis.
By Lemma \ref{A2M}, $\tilde{A}_2$ is a function of
zero exponential type and of Hamburger class. Hence, 
$|\tilde{A}_2(iy)|\to \infty$ as $|y|\to\infty$. On the other
hand, $\sup_y|f(iy)|/|\tilde{A}_1(iy)|<\infty$. Hence, we
have the interpolation formula for $f/\tilde{A}$.
\end{proof}
\medskip
\noindent
{\bf End of the first proof of Theorem \ref{ordering}.}
As usual, we translate the problem to the equivalent 
localization problem in the associated de Branges space $\ho = A\ho(T, \mu)$.
By Lemma \ref{TfLemma}, we may assume that $f$ and $g$ have
only real zeros. Let us consider two subspaces of $\ho$
defined by 
$$
\mathcal{F}_1:=\closspan_{\ho}\biggl{\{}
\frac{f(z)}{z-\lambda}\biggr{\}}_{\lambda\in\mathcal{Z}_f},
\qquad
\mathcal{F}_2:=\closspan_{\ho}
\biggl{\{}\frac{g(z)}{z-\lambda}\biggr{\}}_{\lambda\in\mathcal{Z}_g}.
$$
By Lemma \ref{subconstr} these are de Branges subspaces of $\ho$ and, by 
de Branges' Ordering Theorem (see \cite[Theorem 35]{br})
we conclude that either $\mathcal{F}_1\subset\mathcal{F}_2$ or
$\mathcal{F}_2\subset\mathcal{F}_1$. Assume that
$\mathcal{F}_2\subset\mathcal{F}_1$. The de Branges subspace
$\mathcal{F}_1$ has the localization property. From the inclusion
$g\in\mathcal{F}_1$ we conclude that zeros of $g$ are localized
near the $\mathcal{Z}_{\tilde{A}}$, where $\tilde{A}$ is
the  corresponding $A$-function of the space $\mathcal{F}_1$. Using 
Lemma \ref{subconstr} we get that $T_g\subset T_f$. If
$\mathcal{F}_1\subset\mathcal{F}_2$, we get
$T_f\subset T_g$ by the same arguments. This completes the first proof of Theorem
\ref{ordering}.}
\qed

%********************************************************

\subsection{Second proof of Theorem \ref{ordering}}.
By Lemma \ref{TfLemma} we may assume, in what follows, that
$f=A_1$, $g=\tilde{A_1}$, where $\mathcal{Z}_{A_1},
\mathcal{Z}_{\tilde{A_1}}\subset T$. Thus, we may write
$A=A_1A_2=\tilde{A_1}\tilde{A_2}$ for some entire functions $A_2$
and $\tilde{A_2}$.

Let $A_1=BA_0$, $\tilde{A_1}=\tilde{B}A_0$, where $B$ and
$\tilde{B}$ have no common zeros. To prove Theorem \ref{ordering},
we need to show that either $B$ or $\tilde{B}$ has finite number
of zeros.

{ The following proposition will play a crucial role in the second
proof of Theorem \ref{ordering}.  
It seems to be of certain independent interest.
For the definition of the Smirnov class $\mathcal{N}_+(\mathbb{C}^+)$
see Subsection \ref{prel}.

\begin{proposition}
Let $B$ and $\tilde B$ be finite order 
entire functions, which are real on $\RR$, have only real 
and simple zeros such that both $\frac{B}{\tilde{B}}$
and $\frac{\tilde{B}}{B}$ belong to $\mathcal{N}_+(\mathbb{C}^+)$.
If $\mathcal{Z}_B\cup\mathcal{Z}_{\tilde{B}}$ is a power separated set, 
then for some $M>0$ at least one of the following two statements holds: 

\begin{enumerate}
\begin{item}
there exists a subsequence $\{t_{n_k}\}\subset \mathcal{Z}_B$ such
that $|B'(t_{n_k})|\leq 4|t_{n_k}|^M|\tilde{B}(t_{n_k})|$
\end{item}

\begin{item}
there exists a subsequence
$\{t_{n_k}\}\subset\mathcal{Z}_{\tilde{B}}$ such that
$|\tilde{B}'(t_{n_k})|\leq 4|t_{n_k}|^M|B(t_{n_k})|$.
\end{item}

\end{enumerate}
\label{ordlemma}
\end{proposition}  }

\begin{proof} Assume the converse.
Then for any $M>0$ we have $|B'(t_n)|\geq|t_n|^M|\tilde{B}(t_n)|$,
$t_n\in \mathcal{Z}_{B}$ and an analogous estimate for
$|\tilde{B}'(t_n)|$, $t_n\in\mathcal{Z}_{\tilde{B}}$. Then
the function
$$
F_1(z):=\frac{\tilde{B}(z)}{B(z)}-\sum_{t_n\in\mathcal{Z}_B}
\frac{\tilde{B}(t_n)}{B'(t_n)(z-t_n)}
$$
is an entire function of zero exponential type
(note that by our assumptions on $B$ and $\tilde B$
the series on the right converges for any 
$z\in\mathbb{C}\setminus\mathcal{Z}_B$). The same is true
for the function
$$
F_2(z):=\frac{B(z)}{\tilde{B}(z)}-\sum_{t_n\in\mathcal{Z}_{\tilde{B}}}
\frac{B(t_n)}{\tilde{B}'(t_n)(z-t_n)}.
$$
Thus, for any $z\in \co$, we have the equality
\begin{equation}
\biggl{(}F_1(z)+\sum_{t_n\in\mathcal{Z}_B}\frac{\tilde{B}(t_n)}{B'(t_n)(z-t_n)}\biggr{)}
\cdot\biggl{(}F_2(z)+\sum_{t_n\in\mathcal{Z}_{\tilde{B}}}\frac{B(t_n)}{\tilde{B}'(t_n)(z-t_n)}\biggr{)}=1.
\label{FF}
\end{equation}
From this we conclude that for some $K>0$,
$$
\min(|F_1(z),|F_2(z)|)\lesssim 1, \qquad z\in 
\mathbb{C}\setminus\cup_nD(t_n,|t_n|^{-K}).
$$ 
Let us consider the
circles $C_n=\{w: |w-t_n|=|t_n|^{-K}\}$. We can choose $K$ so large
that both Cauchy transforms from \eqref{FF} are bounded on these
circles. Also, it follows from simple estimates of the Hadamard canonical 
products that $|B(x)| \asymp |B(y)|$ when $x,\, y\in C_n$,
and analogously $|\tilde{B}(x)| \asymp |\tilde{B}(y)|$. 
So, if $\min_{w\in C_n}(|F_1(w)|,|F_2(w)|)\lesssim1$, 
then $|F_1(w)|\lesssim 1$ or $|F_2(w)|\lesssim1$, $w\in C_n$. 
Using the maximum principe for
$F_1$ or $F_2$ we get that 
\begin{equation}
\min(|F_1(z),|F_2(z)|)\lesssim 1, \qquad z\in\mathbb{C}.
\label{brl}
\end{equation}

{  It is a deep result by de Branges \cite[Lemma 8]{br} that if two entire functions 
$F_1$ and $F_2$ of zero exponential type satisfy \eqref{brl}, then 
either $F_1$ or $F_2$ is a constant function. Then from equation 
\eqref{FF} we get that both $F_1$ and 
$F_2$ are non-zero constants.

Note that the pair of functions $z\tilde{B}(z)$ and $B(z)$ also satisfies the conditions of Proposition \ref{ordlemma}.
Repeating the above arguments, we get that the function
 $$\tilde{F}_1(z):=\frac{z\tilde{B}(z)}{B(z)}-\sum_{t_n\in\mathcal{Z}_B}\frac{t_n\tilde{B}(t_n)}{B'(t_n)(z-t_n)}$$ is a non-zero constant. Hence, $U(z):=zF_1(z)-\tilde{F}_1(z)$ is a linear function and $U(iy)=o(|y|)$, $y\rightarrow\infty$. This contradiction proves  Proposition \ref{ordlemma}. }
\end{proof}

%**************************************************************

\subsection{End of the second proof of Theorem \ref{ordering}} Assume
that (i) in Proposition \ref{ordlemma} holds. Dividing if necessary
$B$ by a polynomial we may assume that
$|B'(t_{n_l})|\leq|\tilde{B}(t_{n_l})|$. Hence, we may construct a
lacunary canonical product $U_1$ such that
$\mathcal{Z}_{U_1}\subset\mathcal{Z}_B$ and
$$|B'(t_n)|\leq|\tilde{B}(t_n)|,\qquad t_n\in\mathcal{Z}_{U_1}.$$
Let $U_2$ be another lacunary product with zeros in $\{\Im z\geq
1\}$ such that
\begin{equation}
|U_2(t_n)|=o(|U_1(t_n)|),\qquad n\rightarrow\infty,\qquad t_n\in
T\setminus\mathcal{Z}_{U_1}, \label{Ueq}
\end{equation}
\begin{equation}
|U_2(t_n)|=o(|U_1'(t_n)|),\qquad t_n\in T. \label{Ueq2}
\end{equation}
This may be achieved if we choose zeros of $U_2$ to be much sparser
than the zeros of $U_1$. Let us show that in this case
$$
f:=A_1\cdot\frac{U_2}{U_1}\in\he,
$$
which contradicts the localization. Since $A_1$ is in $\he$, while
$U_1$ and $U_2$ are lacunary products, it is clear that $f/
E$ and $f^*/ E$ are in the Smirnov class
$\mathcal{N}_+(\mathbb{C}^+)$ and $|f(iy)/E(iy)|\rightarrow0$, $|y|\rightarrow\infty$. 
To apply Theorem \ref{t26}, it remains to show that
$$
\sum_{t_n\in T}\frac{|f(t_n)|^2}{|A'(t_n)|^2\mu_n}<\infty.
$$
Since $f$ vanishes on $\mathcal{Z}_{A_1}\setminus\mathcal{Z}_{U_1}$, 
we need to estimate
the sums over $\mathcal{Z}_{A_2}$ and $\mathcal{Z}_{U_1}$.
By \eqref{Ueq} and Lemma \ref{A2M}, we have
$$
\sum_{t_n\in\mathcal{Z}_{A_2}}\frac{|f(t_n)|^2}{|A'(t_n)|^2\mu_n}
=\sum_{t_n\in\mathcal{Z}_{A_2}}\frac{|U_2(t_n)|^2}{|U_1(t_n)|^2}
\cdot\frac{1}{|A'_2(t_n)|^2\mu_n}<\infty
$$
To estimate the sum over
$\mathcal{Z}_{U_1}$ note first that $BA_0$ divides
$A=\tilde{B}A_0\tilde{A}_2$, whence $B$ divides $\tilde{A_2}$.
Thus $\mathcal{Z}_{U_1}\subset\mathcal{Z}_{\tilde{A_2}}$. Also,
for $t_n\in\mathcal{Z}_{U_1}$,
\begin{equation}
|A'_1(t_n)|=|B'(t_n)|\cdot|A_0(t_n)|\leq|A_0(t_n)|\cdot|\tilde{B}(t_n)|=|\tilde{A}_1(t_n)|.
\label{Ueq3}
\end{equation}
Now by \eqref{Ueq}, \eqref{Ueq2}, \eqref{Ueq3} and Lemma \ref{A2M}
applied to $\tilde{A_2}$ we have
$$
\sum_{t_n\in\mathcal{Z}_{U_1}}\frac{|f(t_n)|^2}{|A'(t_n)|^2\mu_n}=
\sum_{t_n\in\mathcal{Z}_{U_1}}
\frac{|U_2(t_n)|^2}{|U_1'(t_n)|^2}\cdot\frac{|A'_1(t_n)|^2}{|\tilde{A}_1(t_n)|^2|\tilde{A}_2(t_n)|^2\mu_n}
$$
$$
\leq \sum_{t_n\in\mathcal{Z}_{U_1}}\frac{1}{|\tilde{A}_2(t_n)|^2\mu_n}<\infty.
$$
Thus, $f\in\he$ and this contradiction completes the proof of Theorem
\ref{ordering}. \qed

%%%%%%%%%%%%%%%%%%%%%%%%%%%%%%%%%%%%%%%%%%%%%%%%%%%%%%%%%%%%%%
%%%%%%%%%%%%%%%%%%%%%%%%%%%%%%%%%%%%%%%%%%%%%%%%%%%%%%%%%%%%%%

\section{Description of spaces having localization property of type 2}
In this section we prove Theorem \ref{type2}.

%**************************************************************

\subsection{Sufficiency\label{2suf}} 
Let $A$ be a function real on
$\mathbb{R}$ with simple real zeros in $T$ and $A=A_1A_2$, where
$A_2$ is the Hamburger class function from (i).
Let
$\he=A\mathcal{H}(T,\mu)$ be the associated de Branges space and
%It
%follows from \eqref{type2eq2}, that the function $A_1$ belongs to
%$\he$. Thus both $T$ and $T_1$ are attraction sets for $\he$ (or,
%equivalently for $\mathcal{H}(T,\mu)$).
let $\mathcal{H}(E_2)$ be a de Branges space constructed from
$T_2$, $\mu|_{T_2}$, i.e.,
$\mathcal{H}(E_2)=A_2\mathcal{H}(T_2,\mu|_{T_2})$.

{  By the hypothesis, the orthogonal complement $\mathcal{L}$ to the polynomials 
in $L^2(T_2,\mu|_{T_2}) =\ell^2(\mu|_{T_2})$ is finite-dimensional. 
If $\{d_n\} \in \ell^2(\mu|_{T_2})\setminus \mathcal{L}$, then there exists
a nonzero moment for the sequence $\{d_n\}$, that is, 
$\sum_{t_n \in T_2} \mu_n d_n t_n^K \ne 0$ for some $K\in\mathbb{N}_0$.
If $f(z) = \sum_{t_n \in T_2} \frac{\mu_nd_n}{z-t_n}$ is the corresponding
function from $\ho(T_2, \mu|_{T_2})$, then, analogously to the arguments from
Subsection \ref{pd1}, we obtain that for any $M>0$ the function $f$
has a zero in $D(t_n, |t_n|^{-M})$, $t_n \in T_2$, 
when $n$ is sufficiently large. Thus, for any function in 
$\ho(T_2, \mu|_{T_2})$ except some finite-dimensional subspace, its zeros
are localized near the whole set $T_2$. 

Now let $\mathcal{G}$ be the subspace of the de Branges space 
$\mathcal{H}(E_2)$ defined by 
$$
\mathcal{G} = \bigg\{A_2\sum_{t_n \in T_2} \frac{\mu_nd_n}{z-t_n}:\, \{d_n\}
\in \mathcal{L} \bigg\}.
$$
This is a finite-dimensional subspace of $\mathcal{H}(E_2)$ 
and it is easy to see that $F\in \mathcal{G}$ if and only if
$F \in \mathcal{H}(E_2)$ and $|F(iy)/A(iy)| = o(|y|^{-M})$, $|y|\to \infty$,
for any $M>0$. It follows that $\mathcal{G}$ satisfies the axioms of a de Branges 
space and, hence, is a de Branges subspace of $\mathcal{H}(E_2)$.
Since $\mathcal{G}$ is finite-dimensional, it consists of the functions 
of the form $SP$ where $S$ is some fixed zero-free function which is real on $\RR$
and $P$ is any polynomial up to some fixed degree $L$. Replacing $A$ by $A/S$ we may 
assume that $\mathcal{G}$ consists of polynomials. }

Thus we conclude that $\mathcal{H}(E_2)$ has the localization
property and for any $F\in\mathcal{H}(E_2)$ we either have
$T_F=\emptyset$ (i.e., $F$ is a polynomial) or $T_F=T_2$.

Let $f\in\mathcal{H}(T,\mu)$, $f(z)=\sum_{t_n\in
T}\frac{c_n\mu^{1/2}_n}{z-t_n}$. Since,
$|A_2(t_n)|\mu^{1/2}_n$ tends to zero faster than any power of
$t_n\in T_1$ when $|t_n|\to \infty$, we have
\begin{equation}
A_2(z)\sum_{t_n\in
T_1}\frac{c_n\mu^{1/2}_n}{z-t_n}=\sum_{t_n\in
T_1}\frac{A_2(t_n)c_n\mu^{1/2}_n}{z-t_n}+H(z) \label{Heq}
\end{equation}
for some entire function $H$ (note that the residues on the left
and the right coincide). Moreover, it is easy to see (using Theorem \ref{t26}) 
that $H\in\mathcal{H}(E_2)$. Indeed, we need only to verify that
$\frac{|H(iy)|}{|A_2(iy)|}\rightarrow0$, $|y|\rightarrow\infty$
(which is obvious since the Cauchy transforms tend to zero, while
$|A_2(iy)|\rightarrow\infty$, $|y|\rightarrow\infty$ for the 
Hamburger class functions), and
$$\sum_{t_n\in T_2}\frac{|H(t_n)|^2}{|A'_2(t_n)|^2\mu_n}<\infty,$$
which is true since the Cauchy transform on the right hand side of
\eqref{Heq} is bounded on $T_1$.

Note also that $F(z):=A_2(z)\sum_{t_n\in
T_2}\frac{c_n\mu^{1/2}_n}{z-t_n}$ is in $\mathcal{H}(E_2)$ by
its definition. Thus,
$$
A_2(z)f(z)=\sum_{t_n\in T_1}\frac{c_nA_2(t_n)\mu^{1/2}_n}{z-t_n}+H(z)+F(z),
$$
with $H+F\in\mathcal{H}(E_2)$. Put
$$
g(z)=\sum_{t_n\in T_1}\frac{c_nA_2(t_n)\mu^{1/2}_n}{z-t_n}.
$$
Assume that $H+F\neq0$. Then, either, $H+F$ is a polynomial or the
zeros of $H+F$ are localized near $T_2$ up to a finite set and,
thus, $|H(t_n)+F(t_n)|>|A_2(t_n)|\cdot|t_n|^{-M}$ for some $M>0$
whenever $t_n\in T_1$ is sufficiently large. In each of the cases
$|H(t_n)+F(t_n)|\gtrsim1$ for $z\in D(t_k,r_k)$, $t_k\in T_1$,
$r_k=|t_k|^{-K}$, with $K$ sufficiently large. { Since, 
$|g(z)|\to 0$ whenever $|z-t_k|=r_k$ and $k\rightarrow\infty$,} we
conclude by the Rouch\'e theorem that $g+H+F$ has exactly one zero
in each $D(t_k,r_k)$, $t_k\in T_1$, except possibly a finite
number. Thus, $f$ has zeros near the whole set $T_1$ and also near
$T_2$ if $H+F$ is not a polynomial (again apply the Rouch\'e
theorem to small disks $D(t_k,r_k)$, $t_k\in T_2$,
$r_k=|t_k|^{-K}$, and use the fact that $|H+F|\gtrsim1$,
$|z-t_k|=r_k$).

It remains to consider the case $H+F=0$. Since  the polynomials
are dense in $L^2(T,\tilde{\mu})$, the space
$\mathcal{H}(T,\tilde{\mu})$ has the strong localization property,
and so $T_g=T_1$ up to a finite set. Also, $Af$ vanishes on $T_2$
and we conclude that $T_f=T$.

%**************************************************************

\subsection{Necessity\label{2nec}} Assume that $\mathcal{H}(T,\mu)$ has the
localization property of type $2$ and let
$\he=A\mathcal{H}(T,\mu)$ be some associated de Branges space. Let
$f$ be a function from $\mathcal{H}(T,\mu)$ such that
$\#(T\setminus T_f)=\infty$. Then, by Lemma \ref{TfLemma} there
exists $T_1$ ($T_1=T_f$ up to a finite set) such that there exists
a function $A_1$ with simple zeros in $T_1$ and $A_1\in\he$. We
now may write $A=A_1A_2$ for some entire $A_2$ with
$\mathcal{Z}_{A_2}=T_2$. The necessity of condition (i) 
follows from Lemma \ref{A2M}. 
 Note also that by our hypothesis
(localization of type $2$) zeros of any function $F \in\he$ are
localized near $T$ or near $T_1$.

Hence,  zeros of any function $F$ of the form 
$A_2(z)\sum_{t_n\in T_2}\frac{c_n\mu^{1/2}_n}{z-t_n}$, $\{c_n\}\in\ell^2$,
either form a finite set ($F$ is a polynomial) or 
are localized near $T_2$. It remains to prove that  the 
degrees of all polynomials $P\in\mathcal{H}(E_2)$ are uniformly bounded.

Let us show that the
property that $PA_1\in\he$ for any polynomial $P$ contradicts the
localization property. Indeed, let $\mathcal{H}_0$ be the de
Branges subspace of $\he$ constructed as in Lemma
\ref{subconstr}, namely,
$$\mathcal{H}_0:=\closspan_{\he}\biggl{\{}\frac{A_1(z)}{z-t_n}\biggr{\}}_{t_n\in T_1}.$$
Moreover, fix a sequence of polynomials $P_k$ with simple real
zeros disjoint with $T$, and put
$$\mathcal{H}_k:=\closspan_{\he}\biggl{\{}\frac{P_k(z)A_1(z)}{z-t_n}\biggr{\}}_{t_n\in T_1\cup\mathcal{Z}_{P_k}}.$$
Each $\mathcal{H}_k$ is a de Branges subspace of $\he$. Let us
show that
\begin{equation}
\dim(\mathcal{H}_{k+1} \ominus \mathcal{H}_k)=1,\qquad \text{ for any
} k\geq 0. \label{keq}
\end{equation}
If \ref{keq} is proved, then we conclude that the space
$\tilde{\mathcal{H}}:=\closspan(\cup_k\mathcal{H}_k)$ 
contains no subspace of codimension $1$, 
a contradiction to Theorem \ref{chain}. To prove
\eqref{keq} note that for any function from $\mathcal{H}_{k-1}$ of
the form $F(z):=\frac{P_{k-1}(z)A_1(z)}{z-\lambda}$, $\lambda\in
T_1\cup\mathcal{Z}_{P_{k-1}}$, we have $zF(z)\in\mathcal{H}_k$.
Since $\mathcal{H}_k$ contains a subspace of codimension $1$ we
conclude that $\mathcal{H}_{k-1}$ is of codimension $1$ in
$\mathcal{H}_k$.

Thus, (i) and (ii) are proved. Assume that (iii) is not
satisfied, that is, the polynomials, are not dense in
$\mathcal{H}(T_1,\tilde{\mu})$, and so this space does not have
the strong localization property. Choose $g(z)=\sum_{t_n\in
T_1}\frac{c_nA_2(t_n)\mu^{1/2}_n}{z-t_n}\in\mathcal{H}(T_1,\tilde{\mu})$
such that the zeros of $g$ are not localized near $T_1$ (which
means that there exists an infinite sequence of disks
$D(t_{n_j},|t_{n_j}|^{-M})$, $t_{n_j}\in T_1$, such that
$\#D(t_{n_j},|t_{n_j}|^{-M})\cap \mathcal{Z}_g=0$). Now put,
$$H(z)=A_2(z)\sum_{t_n\in T_1}\frac{c_n\mu^{1/2}_n}{z-t_n}-g(z).$$
The function $H$ is entire and, as in the proof of sufficiency,
$H\in\mathcal{H}(E_2)$, where
$\mathcal{H}(E_2)=A_2\mathcal{H}(T,\mu|_{T_2})$. This means
that $H$ can be written as
$$H(z)=-A_2(z)\sum_{t_n\in T_2}\frac{c_n\mu^{1/2}_n}{z-t_n}\quad
\text{ for some } \{c_n\}\in\ell^2.$$
Now put
$$
f(z)=\sum_{t_n\in T}\frac{c_n\mu^{1/2}_n}{z-t_n}.
$$
By the construction,
$$
f(z)=\frac{H(z)}{A_2(z)}+\frac{g(z)}{A_2(z)}-\frac{H(z)}{A_2(z)}=\frac{g(z)}{A_2(z)}
$$
and $Af=A_1g$. However, the zeros of $A_1g$ are not localized near
the whole $T_1$, a contradiction.

%%%%%%%%%%%%%%%%%%%%%%%%%%%%%%%%%%%%%%%%%%%%%%%%%%%%%%%%%%
%%%%%%%%%%%%%%%%%%%%%%%%%%%%%%%%%%%%%%%%%%%%%%%%%%%%%%%%%%

%%%%%%%%%%%%%%%%%%%%%%%%%%%%%%%%%%%%%%%%%%%%%%%%%%%%%%%%%%
%%%%%%%%%%%%%%%%%%%%%%%%%%%%%%%%%%%%%%%%%%%%%%%%%%%%%%%%%%
\subsection{Localization of type $N$} Using the arguments 
from the proof of Theorem \ref{type2} we can find the analogous description of the spaces with localization property of type $N$.

\begin{theorem}
\label{Nloc}
Let $\mathcal{H}(T,\mu)$ be a space with power separated $T$. 
The space $\mathcal{H}(T,\mu)$  has the localization property of type $N$
if and only if there exist sets $W_j\subset T$, $1\leq j\leq N$, such that

\begin{enumerate}
\begin{item}
$W_1\subset W_2\subset...\subset W_N=T$ and 
$\#(W_{j+1}\setminus W_j)=\infty$\textup;
\end{item}

\begin{item}
There exist entire functions $B_j$, $1\leq j < N$ of zero exponential type such that $\mathcal{Z}_{B_j}=T\setminus W_j$, (in particular, $B_N\equiv const$) and entire functions $B_{j+1}/ B_j$, $1\leq j \leq N-1$, are of Hamburger class\textup;
\end{item}

\begin{item}
The polynomials belong to the spaces $L^2(W_{j+1}\setminus W_j, 
\sum_{t_n\in W_{j+1}\setminus W_j}|B_{j+1}(t_n)|^2\mu_n\delta_{t_n})$, 
$1\leq j \leq N-1$, they are not dense there,  
but their closure is of finite codimension in these spaces\textup;
\end{item}

\begin{item}
The polynomials belong to the space $L^2(W_1, 
\sum_{t_n\in W_1} |B_{1}(t_n)|^2\mu_n\delta_{t_n})$
and are dense there.
\end{item}
\end{enumerate}

Moreover, in this case, for any nonzero $f\in \mathcal{H}(T,\mu)$ its attraction set $T_f$ 
coincides with one of the sets $W_j$, $1\leq j\leq N$.
\end{theorem}

\begin{proof}
The sufficiency of the conditions (i)--(iv) can be obtained by induction on $N$. It is easy to check that $|B_{N-1}(t_n)|\mu^{1\slash2}_n=o(|t_n|^M)$,  $t_n\in W_{N-1}$, $|t_n|\rightarrow\infty$ for any $M$.
Put
$$\mathcal{H}:=\biggl{\{}\sum_{t_n\in W_{N-1}}\frac{c_nB_{N-1}(t_n)\mu^{1\slash2}_n}{z-t_n}: \{c_n\}\in\ell^2\biggr{\}}.$$
From the induction hypothesis we know that $\mathcal{H}$ has the localization property of type $N-1$. Now we can repeat the arguments from the Subsection \ref{2suf}.

The necessity of conditions (i) and (ii) can be derived from Lemmas \ref{subconstr} and \ref{A2M} respectively. The necessity
of conditions (iii) and (iv) can be obtained by induction on $N$ using the arguments from Subsection \ref{2nec}.
\end{proof}

In contrast to the case of the localization of type $N$, 
the spaces with localization near infinitely many attraction sets may have
a very complicated structure since the indivisible intervals can accumulate
on the left in many different ways. It is an interesting problem to
describe these sets analytically (say, as zero sets of entire 
functions from some special classes).

%**************************************************************

\end{document}